\documentclass[10pt]{amsart}
\usepackage{egastyle}

\usepackage{amssymb,amsmath,amscd,amsthm,amsfonts,wasysym,mathrsfs,amsxtra,stmaryrd,array}

\usepackage{graphicx,array,url}
\usepackage[vcentermath]{genyoungtabtikz}
\usepackage{tikz}
\usepackage[section]{placeins}
\usepackage{afterpage}
\usepackage{verbatim}
\usepackage{tikz-cd}
\input xy
\xyoption{all}

\usepackage{algorithm}
\usepackage{algorithmic}

\newcommand{\OO}{\mathcal{O}}

\newcommand{\image}{\textnormal{im}\,}

\newcommand{\Hom}{\textnormal{Hom}}
\newcommand{\dimension}{\textnormal{dim}\,}

\newcommand{\supp}{\textnormal{supp}}
\newcommand{\Ext}{\textnormal{Ext}}

\newcommand{\Bc}{\mathcal{B}}

\newcommand{\Fc}{\mathcal{F}}
\newcommand{\Tc}{\mathcal{T}}

\newcommand{\Coh}{\mathrm{Coh}}

\newcommand{\arinj}{\ar@{^{(}->}}
\newcommand{\arsurj}{\ar@{->>}}
\newcommand{\areq}{\ar@{=}}

\newcommand{\ch}{\mathrm{ch}}

\newcommand{\bsm}{\begin{smallmatrix}}
\newcommand{\esm}{\end{smallmatrix}}

\newcommand{\R}{\mathbb{R}}

\newcommand{\B}{\mathcal{B}}
\newcommand{\T}{\mathcal{T}}
\newcommand{\F}{\mathcal{F}}

\newcommand{\HN}{\text{HN}}

\newcommand{\cP}{\mathcal{P}}

\newcommand{\cO}{\mathcal{O}}

\newcommand{\NS}{\text{NS}}

\makeatletter
\newtheorem*{rep@theorem}{\rep@title}
\newcommand{\newreptheorem}[2]{%
\newenvironment{rep#1}[1]{%
 \def\rep@title{#2 \ref{##1}}%
 \begin{rep@theorem}}%
 {\end{rep@theorem}}}
\makeatother

  {\begin{genthm}{Question}{true}{#1}{paragraph}}%
  {\end{genthm}}

  {\begin{genthm}{Question}{true}{#1}{equation}}%
  {\end{genthm}}

\newreptheorem{theorem}{Theorem}

\usepackage{scalerel,stackengine}
\stackMath
\newcommand\reallywidehat[1]{%
\savestack{\tmpbox}{\stretchto{%
  \scaleto{%
    \scalerel*[\widthof{\ensuremath{#1}}]{\kern-.6pt\bigwedge\kern-.6pt}%
    {\rule[-\textheight/2]{1ex}{\textheight}}
  }{\textheight}%
}{0.5ex}}%
\stackon[1pt]{#1}{\tmpbox}%
}
\begin{document}

\title[]{Stability of the cotangent bundle on surfaces with ample canonical class}

\author[Tristan C. Collins]{Tristan C. Collins}
\address{Department of Mathematics \\
University of Toronto\\
40 St. George St. \\
Toronto, ON, \\
Canada}
\email{tristanc@math.toronto.edu}

\author[Jason Lo]{Jason Lo}
\address{Department of Mathematics \\
California State University, Northridge\\
18111 Nordhoff Street\\
Northridge CA 91330 \\
USA}
\email{jason.lo@csun.edu}

\author[Yun Shi]{Yun Shi}
\address{Department of Mathematics \\
Missouri State University\\
901 S. National Ave.
Springfield, MO 65897\\
USA}
\email{ys694e@@missouristate.edu}

\author[Shing-Tung Yau]{Shing-Tung Yau}
\address{Yau Mathematical Sciences Center\\
Tsinghua University\\
Haidian District, Beijing\\
China}
\email{styau@tsinghua.edu.cn}


\begin{abstract}
In this article, we study the Bridgeland stability of the cotangent bundle on K3 surfaces and surfaces with ample canonical class.  On a surface $X$ with ample canonical class $K$, we find values $t_0>0$ such that the cotangent bundle is semistable with respect to the Bridgeland stability $\sigma_{tK}$ for all $t > t_0$.  We also study special cases such as the fake projective planes  and the fake quadrics.
\end{abstract}

\maketitle
\tableofcontents

\section{Introduction}\label{sec:intro}

BPS branes in type IIB string theory can be described either as vector bundles admitting hermitian metrics solving a nonlinear partial differential equation called the deformed Hermitian-Yang-Mills or LYZ equation, or alternatively as stable objects in a Bridgeland stability condition. Let us recall some basic aspects of these two descriptions.  Let $E\rightarrow (X,\omega)$ be a holomorphic vector bundle over a compact K\"ahler manifold $X$ with $\dim_{\mathbb{C}}X=n$.  The LYZ equation seeks a hermitian metric $h$ on $E$ such that the curvature $F_h$ of the Chern connection on $E$ satisfies the partial differential equation
\begin{equation}\label{eq: dHYMIntro}
\begin{aligned}
    {\rm Im}\left(e^{-i\hat{\theta}}(\omega \otimes I_{E} -F_h)^n\right)&=0\\
    {\rm Re}\left(e^{-i\hat{\theta}}(\omega \otimes I_{E} -F_h)^n\right)&>0
\end{aligned}
\end{equation}
where the real and imaginary parts are defined using the hermitian structure induced by $h$.  Upon taking the ``large radius limit", that is setting $\omega_t=t\omega$ and taking $t\rightarrow +\infty$ we find that~\eqref{eq: dHYMIntro} reduces to
\begin{equation}\label{eq: HYM}
i\Lambda F_{h} = \mu(E)I_{E} + o(t^{-1}).
\end{equation}
That is, in the large radius limit the nonlinear system~\eqref{eq: dHYMIntro} reduces to the Hermtian-Yang-Mills equation.  The system~\eqref{eq: dHYMIntro} was derived independently (for vector bundles with abelian gauge groups) by Mari\~no-Minasian-Moore-Strominger \cite{MMMS} from the action functional point-of-view, and by Leung-Yau-Zaslow \cite{LYZ} by applying a real Fourier-Mukai transform to a special Lagrangian in the mirror type IIA theory. 

By the Donaldson-Uhlenbeck-Yau Theorem \cite{DonaldsonASDYMC, UYHYMC}, holomorphic vector bundles admitting solutions of~\eqref{eq: HYM} can alternatively be characterized algebraically by the notion of slope stability.  Inspired in part by this correspondence, Douglas-Fiol-R\"omelsberger \cite{DFR} introduced the notion of $\Pi$-stability as a purely algebraic description of BPS $D$-branes in type IIA.  The theory of $\Pi$-stability was subsequently developed into the general notion of a stability condition on a triangulated category by Bridgeland \cite{StabTC}.  These two points of view lead naturally to conjecture that a holomorphic vector bundle $E\rightarrow (X,\omega)$ admits a solution of the LYZ equation~\eqref{eq: dHYMIntro} if and only if it is Bridgeland stable as an object in $D^{b}{\rm Coh}(X)$.  This conjecture can be viewed as a holomorphic version of the Thomas-Yau-Joyce conjecture \cite{Th, ThY, Joyce}.  The central charge relevant to the LYZ equation is given by
\[
Z(E)= -\int_{X}e^{-\frac{i}{2\pi}\omega}ch(E).
\]

The theory of deformed Hermitian-Yang-Mills equation, and its connections with notions of stability, has been developed mostly for the case of line bundles \cite{Ballal, ChanJacob, GaoChen, ChuLee, ChuLeeTak, CollinsYau, CollinsYaugeo, CCL, CXY, CJY, CoShdHYM, CLSY, DatarPing, DatMetSong, Jacob, JacobSheu, JacobYau, Mete, Murakami, Murakami2, Song, Sto25, Tak1, Tak2}.  In several settings, the existence of a solution to the deformed Hermitian–Yang–Mills equation for a line bundle implies its stability with respect to standard Bridgeland stability conditions. For certain surfaces of Picard rank two, this follows directly from a result in \cite{ARCARA20161655}, \cite{MY25} extends the implication to a del Pezzo surface of Picard rank three using similar methods. For certain line bundles on a Weierstra\ss{} elliptic K3 surface of higher Picard rank, \cite{CLSY} proves the implication using the symmetry of the derived category, i.e. derived autoequivalences. The converse fails in general, as shown by the counterexample in \cite{CoShdHYM}. 
For applications of the theory of the LYZ equation on Bridgeland stability of line bundles, see \cite{Ant26}.
 For higher-rank bundles on smooth projective surfaces $X$ very little is known, both from the analytic and algebraic points of view.  Dervan-McCarthy-Sektnan \cite{DMS} proved the existence of solutions sufficiently close to the large radius limit provided $E$ is slope semi-stable by a perturbation argument, building on earlier related work of Leung \cite{Leung97}; see also \cite{DNST, LaraSaEarp} for related recent work.  Solutions in homogeneous settings, where the construction of solutions to~\eqref{eq: dHYMIntro} can be reduced to an algebraic problem, have been studied by Correa \cite{Correa} and  Charbonneau-Oliveira-Sena-Dias \cite{COSD}. Necessary algebraic conditions for existence of solutions to the rank $2$ LYZ equation on K\"ahler surfaces were studied by Keller-Scarpa \cite{KellerScarpa}. Recently, the first author and Zhang \cite{CollinsZhang} produced the first inhomogeneous solutions of the system~\eqref{eq: dHYMIntro} on vector bundles of arbitrary rank by studying $SU(2)$-equivariant vector bundles on $X\times \mathbb{P}^1$ where $X$ is a Riemann surface.  In this setting dimensional reduction yields a nonlinear equation for a holomorphic triple $(E_1,E_2,\Phi)$ with $\Phi\in H^{0}(X,{\rm Hom}(E_2,E_1))$.  For $SU(2)$-equivariant vector bundles of rank $2$, this equation was previously studied by Pingali \cite{Pingali}, Takahashi \cite{Takahashi} and Ghosh \cite{Ghosh}.


The starting point for our investigation is the following observation.  Let $X$ be a Calabi-Yau surface, and let $E=\Omega_{X}$ denote the cotangent bundle.  By Yau's solution of the Calabi conjecture \cite{YauMA}, for every K\"ahler class $[\omega]\in H^{1,1}(X,\mathbb{R})$ there is a unique K\"ahler metric $\omega_{CY}$ lying in the class $[\omega]$ and such that ${\rm Ric}(\omega_{CY})=0$. Denote by $h_{CY}$ the hermitian metric on $\Omega_{X}$ induced by $\omega_{CY}$ and let $F_{CY}$ denote the curvature of $h_{CY}$.  Then
\[
{\rm Im}\left(\omega_{CY}\otimes I_{\Omega_{X}} - F_{CY}\right)^2 =0  
\]
using that $\omega_{CY}\wedge F_{CY} \propto {\rm Ric}(\omega_{CY})=0$.  The real part is given by
\[
{\rm Re}\left(\omega_{CY}\otimes I_{\Omega_{X}} - F_{CY}\right)^2= \omega^2_{CY}\otimes I_{\Omega_{X}} + F_{CY}\wedge F_{CY}  \geq \omega^2_{CY}\otimes I_{\Omega_{X}}
\]
 where the inequality follows from the primitivity of $F_{CY}$.  Thus the Calabi-Yau metric solves the deformed Hermitian-Yang-Mills equation.  Thus, it is natural to ask:
\smallskip

\noindent{\bf Question 1}: If $X$ is a Calabi-Yau surface, is $\Omega_{X}$ Bridgeland stable for all polarizations?
\smallskip

Similar considerations apply to surfaces with ample canonical bundle.  In this case the works of Aubin \cite{Aubin} and Yau \cite{YauMA} yield the existence of a K\"ahler-Einstein metric $\omega_{KE}$ with negative scalar curvature in the classes $tc_1(K_{X})$ for $t>0$. The K\"ahler-Einstein metric is also Hermitian-Einstein, and hence solves the LYZ equation in the large radius limit.  However, unlike the Calabi-Yau case, when $K_{X}$ is ample the LYZ equation for $\Omega_{X}$ is not scaling invariant. Hence one expects that there is some threshold value $t_0>0$ such that, for $t \leq t_0$ the nonlinear effects of the quadratic curvature term appearing in the LYZ equation~\eqref{eq: dHYMIntro} may obstruct the existence of a solution with respect to the K\"ahler form $t\omega_{KE}$.  This problem seems complicated to attack analytically, but we can ask the same question in the realm of algebraic geometry:
\smallskip

\noindent{\bf Question 2}: If $X$ is a surface with $K_{X}$ ample, for what values of $t>0$ is $\Omega_{X}$ Bridgeland stable with respect to the K\"ahler class $tK_{X}$?
\smallskip

In this paper we answer these questions in a variety of settings. To begin with, when $X$ is a K3 surface, we show that the tangent bundle $X$ is $\sigma_\omega$-stable for any ample class $\omega$ (Corollary \ref{cor:K3-main}).  This is an easy corollary of results of  Huybrechts on minimal objects in a tilted heart.  This can be seen as the modern analogue of the result that on a smooth projective variety, the tangent bundle $T_X$ is slope semistable with respect to any ample class.

The rest of this paper, and most of the results here, concern smooth projective surfaces $X$ where the ample class $K_X$ is ample.  These are minimal surfaces of general type, and so are related to the geography problem of surfaces of general type.  Our focus is to understand the range of $t>0$ such that the cotangent bundle $\Omega_X$ is $\sigma_{tK_X}$-stable.  Existing results on the boundedness of Bridgeland walls on surfaces tell us that there exists some $t_0>0$ such that $\Omega_X$ is $\sigma_{tK_X}$-stable for all $t >t_0$, and so our primary concern is in finding the least possible value of $t_0$, which corresponds to locating the outermost `mini-wall' for $\Omega_X$ along the ray $\{ \sigma_{tK_X}\}_{t>0}$ in the Bridgeland stability manifold.

Our approach is to first estimate the rank of a destabilizing subobject $A$ of $\Omega_X$ and then, when the rank of $A$ is small, analyze all the components of the destabilizing sequence.  The theory of negative curves invariably becomes involved, and we find results by Miyaoka especially useful.

In our computations,   as the value of the invariant
\[
  C := \tfrac{\tau_X}{2K^2_X} = \tfrac{1}{2} - \tfrac{c_2(X)}{K_X^2}
\]
decreases, our estimate of $t_0$ increases. (Here, $\tau_X=K_X^2-2c_2(X)$ is often referred to as the index of the surface.) We do not know, however, if this means the least possible value of $t_0$ (i.e.\ the actual, outer-most mini-wall) does in fact increase as $C$ decreases.

For an arbitrary smooth projective surface with ample canonical class, our main result is:

\begin{thm}[Theorem \ref{thm:cotangentbld-main-1}]
Let $X$ be a smooth projective surface with ample canonical class $K$ such that the cotangent bundle $\Omega_X$ is twisted $K$-Gieseker semistable.  Then the cotangent bundle $\Omega_X$ is $\sigma_{tK}$-semistable under the following conditions 
\begin{itemize}
    \item[(i)] When $\tau_X \geq 0$:   
    \[
 t^2 >  \tfrac{K^2-1}{4K^2} \left( 4c_2(X) - K^2-1\right); 
\]
in particular, when $t^2 >   \tfrac{K^2}{4}$.
    \item[(ii)] When $\tau_X \leq 0$:  
    \[
    t^2 >  c_2(X)-\tfrac{3}{8}.
    \]
\end{itemize}
\end{thm}

When we specialize to particular surfaces, our results include: 
\begin{itemize}
    \item On a surface $X$ of general type with minimal $c_2(X)$, which must be either the Cartwright-Steger surface or one of the 100 fake projective planes, the cotangent bundle $\Omega_X$ is $\sigma_{tK_X}$-semistable for all $t>0$ (Proposition \ref{pro:CSsurf} and Theorem \ref{thm:AG54-2-1}). 
    \item On minimal surfaces of general type with $p_g=0$ and $K_X^2=8$ - such surfaces are called fake quadrics and have ample canonical class $K_X$ - the cotangent bundle $\Omega_X$ is $\sigma_{tK_X}$-semistable for all $t \geq \frac{1}{\sqrt3}$ (Corollary \ref{cor:fakequadricmain}).
\end{itemize}


\paragraph[Acknowledgements] We thank Fabrizio Catanese,  Cristian Martinez, Jordan Mitas, Rita Pardini, and Jianqiang Yang  for answering our questions and helpful discussions during the completion of this manuscript. T.C.C.~is supported in part by NSERC Discovery grant RGPIN-2024-03853.  J.L. was partially supported by NSF grant DMS-2100906. 

\section{Preliminaries}

In this section, we give some reminders of the standard constructions and notations related to stability conditions.

For a smooth projective variety $X$, we will always write $\Coh (X)$ to denote the category of coherent sheaves on $X$, and $D^b(X)$ to denote $D^b(\Coh (X))$, the bounded derived category of $\Coh (X)$. We will also write $\operatorname{NS}(X)$ to denote the  Néron–Severi group of $X$, and write $\operatorname{NS}_{\mathbb{R}}(X)$ to denote $\operatorname{NS}(X)\otimes_{\mathbb{Z}}\mathbb{R}$.

\subsection{Generalities on Bridgeland stability conditions} Let $X$ be a smooth projective variety.  We briefly recall the definition of a Bridgeland stability condition on $D^b(X)$.  

\begin{defn}\label{defn: slicing}
A {\em slicing} $\cP$ of $D^{b}(X)$ is a collection of full subcategories $\cP(\phi) \subset D^{b}(X)$ for all $\phi \in \mathbb{R}$ such that
\begin{enumerate}
\item $\cP(\phi)[1] = \cP(\phi+1)$,
\item if $\phi_1 > \phi_2$ and $A\in \cP(\phi_1)$, $B \in \cP(\phi_2)$, then ${\rm Hom}(A,B) =0$,
\item every $E\in D^{b}(X)$ admits a Harder-Narasimhan (HN) filtration by objects in $\cP(\phi_i)$ for some $1 \leq i \leq m$. [show the filtration by exact triangles]
\end{enumerate}
\end{defn}
\begin{defn}\label{defn: BrStab}
A \emph{Bridgeland stability condition} on $D^{b}(X)$ with central charge $Z$ is a slicing $\cP$ satisfying the following properties
\begin{enumerate}
\item For any non-zero $E\in \cP(\phi)$ we have
\[
Z(E) \in \mathbb{R}_{>0} e^{\sqrt{-1}\phi},
\]
\item
\[
C := \inf \left\{ \frac{|Z(E)|}{\|\ch(E)\|} : 0 \ne E \in \cP(\phi), \phi \in \mathbb{R} \right\} >0
\]
where $\| \cdot \|$ is any norm on the finite dimensional vector space $H^{\text{even}}(X, \mathbb{R})$.
\end{enumerate}
\end{defn}
If $Z(E) \in \mathbb{R}_{>0} e^{\sqrt{-1}\phi}$ for some $\phi\in\mathbb{R}$, then $\phi$ is called the phase of $E$ and is denoted by $\phi(E)$.

Let $\sigma=(Z,\mathcal{P})$ be a stability condition on $D^b(X)$. A nonzero object $E\in D^b(X)$ is called $\sigma$-semistable if it is an object of $\mathcal{P}(\phi)$ for some $\phi\in\mathbb{R}$.  We will write  $\mathbb{H}$ to denote the upper-half complex plane together with the negative real axis
\[
  \{ re^{i\pi \phi} : r \in \mathbb{R}_{>0}, \phi \in (0,1]\}.
\]

\begin{prop}[Bridgeland, \cite{StabTC}]\label{prop: BrStab}
	\label{pro1}
	A Bridgeland stability condition on $D^b(X)$ is equivalent to the following data: the heart $\mathcal{A}$ of a bounded t-structure on $D^b(X)$, and a central charge $Z: K(\mathcal{A})\rightarrow \mathbb{C}$ such that for every nonzero object $E\in \mathcal{A}$, one has 
 
 (i) $Z(E)\in \mathbb{H}$,
 
 (ii) $E$ has a finite filtration
	\begin{equation*}
	0=E_0\subset E_1\subset...\subset E_{n-1}\subset E_n=E
	\end{equation*} 
	in $\mathcal{A}$ such that each $\HN_i(E)=E_i/E_{i-1}$ is a  semistable object in $\mathcal{A}$, and the  $\HN_i(E)$ have strictly decreasing phases $\phi$. Furthermore, the central charge satisfies Definition \ref{defn: BrStab} (2).
\end{prop}


Using this equivalent characterization of stability, an object $E\in\mathcal{A}$ is called $\sigma$-(semi)stable if, for every proper nonzero subobject $F\subset E$ in $\mathcal{A}$,
\[
\phi(F)<(\leq)\phi(E).
\]
In many situations, it is more convenient to use the Bridgeland slope
\[
\rho(E):=-\frac{\Re Z(E)}{\Im Z(E)}.
\]Then the condition $\phi(F)<(\leq)\phi(E)$ is equivalent to
\[
\rho(F)<(\leq)\rho(E).
\]

\paragraph[Stability conditions on surfaces] Suppose $X$ is a smooth projective surface.  Let $\omega\in \mathrm{NS}_\mathbb{R}(X)$ be an ample class, and $B\in \mathrm{NS}_\mathbb{R}(X)$. For any object $E\in \Coh (X)$, we define 
\begin{equation*}
\mu_{\omega, B}(E)=\frac{\omega \ch_1^B(E)}{\ch_0(E)}
\end{equation*}
where, if $E$ is a torsion sheaf, we set $\mu_{\omega, B}(E)=\infty$. Then for any torsion-free sheaf $E\in \Coh(X)$, $E$ admits a filtration in $\Coh (X)$
\begin{equation*}
	0=E_0\subset E_1\subset E_2 ...\subset E_n=E
\end{equation*}
where each $E_{i+1}/E_i$ is $\mu_{\omega, B}$-semistable and the  $\mu_{\omega, B}(E_{i+1}/E_i)$ are strictly decreasing. The factors $E_{i+1}/E_i$ are called the HN factors of $E$, and are denoted by $\HN_i(E)$.
Consider the following full subcategories of $\Coh(X)$:

\begin{center}
$\mathcal{T}_{\omega, B}=\{E\in \Coh(X)| \mu_{\omega, B}(\HN_i(E))>0 \text{ for any } i\}$, 
\end{center}
\begin{center}
$\mathcal{F}_{\omega, B}=\{E\in \Coh(X)| \mu_{\omega, B}(\HN_i(E))\leq 0 \text{ for any } i\}$. 
\end{center}
Then  $(\T_{{\omega, B}}, \F_{\omega, B})$ is a torsion pair in $\Coh (X)$ and tilting at this torsion pair gives an abelian subcategory of  $D^b(X)$ 
\begin{equation*}
\B_{\omega, B}=\langle \F_{\omega, B}[1], \T_{\omega, B}\rangle.
\end{equation*}
Equivalently, we can describe $\B_{\omega, B}$ as
\begin{equation*}
\B_{\omega, B}= \{E\in D^b(X)| H^0(E)\in \mathcal{T}_{\omega, B},\, H^{-1}(E)\in \mathcal{F}_{\omega, B}, \text{ and }H^i(E)=0 \text{ for } i\neq -1, 0\}.
\end{equation*}

Writing $K(X)$ to denote the Grothendieck group of $D^b(X)$, consider the central charge (a group homomorphism) $Z_{\omega, B} : K(X) \to \mathbb{C}$ given by  
\[
Z_{\omega, B}(E)=-\int e^{-i\omega}\ch^B(E).
\]
\begin{thm}[\cite{SCK3}, \cite{arcara2013minimal}]
The pair $\sigma_{\omega, B}:=(Z_{\omega, B}, \B_{\omega, B})$ defines a Bridgeland stability condition on $D^b(X)$.
\end{thm}

We will consider stability condition on surfaces of the form $\sigma_{\omega, B}$ in this paper. 

When $B=0$, we will drop the subscript $B$ in notations such as $\mu_{\omega, B}, \mathcal{T}_{\omega, B}, \mathcal{F}_{\omega, B}, \mathcal{B}_{\omega, B}, Z_{\omega, B}$, and $\sigma_{\omega, B}$, and simply write $\rho_\omega$ to denote the Bridgeland slope function $\rho (-) = -\Re Z_\omega (-)/\Im Z_\omega (-)$.


\paragraph[Preliminaries of surfaces] \label{para:sgt-basics} For a smooth projective surface $X$, if we write $K_X$ to denote its canonical divisor, $T_X$ to denote its tangent bundle, and $p_a(X)$ to denote its arithmetic genus, then Hirzebruch-Riemann-Roch  \cite[Example 4.1.2, Appendix 4]{Harts} gives
    \begin{equation}\label{eq:HRR}
    1+p_a(X) = \tfrac{1}{12}(K_X^2+c_2(T_X)).
    \end{equation}
A surface of general type with geometric genus $p_g = 0$ must have irregularity $q=0$ (\cite[Remark, p.180]{bombieri1973canonical}; also see \cite[(0.2)]{barlow1985simply}), and hence $p_a=0$.  Also, the adjunction formula says that a smooth curve $C$ of genus $g$ on $X$ satisfies
\begin{equation}\label{eq:adjunction}
2g-2 = C(C+K_X).
\end{equation}



\section{Stability of the cotangent bundle on K3 surfaces}
This paper focuses primarily on surfaces of general type. The K3 case follows easily from Huybrechts’ result, but we include a brief section for completeness.

Let $X$ be an algebraic K3 surface in the sense of \cite[Definition 1.1]{huybrechts2016lectures}, i.e.\ a smooth projective surface over $\mathbb{C}$ with $\omega_X=\OO_X$ and $H^1(X,\OO_X)=0$.  We have the following result on the slope stability of the (co)tangent bundle on K3 surfaces.

\begin{prop}\cite[Prop 4.5, Chap.\ 9]{huybrechts2016lectures}\label{prop:Huy-tbslopestab}
On a polarized K3 surface $(X,H)$ over a field of characteristic 0, the tangent bundle $T_X$ is $\mu_H$-stable.
\end{prop}

On the other hand, we have the following result by Huybrechts  regarding the minimal objects in the heart $\Bc_{\omega, B}$.   An nonzero object in an abelian category $\mathcal{A}$ is called \emph{minimal} if every surjection $E \to E'$ in $\mathcal{A}$ is either an isomorphism or the zero map.


\begin{prop}\cite[Proposition 2.2]{Huy1}\label{prop:Huy-minobjheart}
Let $X$ be a smooth projective surface.  For any divisors $\omega, B$ on $X$ where $\omega$ is ample, the minimal objects in the heart $\mathcal{B}_{\omega, B}$ are precisely the objects
\begin{itemize}
    \item skyscraper sheaves $k(x)$, where $x \in X$ is a closed point, and 
    \item $F[1]$ where $F$ is a locally free $\mu_{\omega, B}$-stable sheaf with $\mu_{\omega, B}(F)=0$.
\end{itemize}
\end{prop}

Using Propositions \ref{prop:Huy-tbslopestab} and \ref{prop:Huy-minobjheart}, we can deduce the Bridgeland stability of the tangent bundle on any K3 surface:

\begin{cor}\label{cor:K3-main}
Let $X$ be a K3 surface.  Then the tangent bundle $T_X$ and the cotangent bundle $\Omega_X$ are $\sigma_\omega$-stable for every ample class $\omega$ on $X$.
\end{cor}

Recall that a simple object $E$ in $D^b(X)$ is an object satisfying $\Hom (E,E)=\mathbb{C}$.

\begin{proof}
Take any ample class $\omega$ on $X$.  By Proposition \ref{prop:Huy-tbslopestab}, the tangent bundle $T_X$ is $\mu_\omega$-stable.  Since $\mu_\omega (T_X)=0$, we have $T_X[1] \in \mathcal{F}_\omega \subset \Bc_\omega$.  Furthermore, $T_X[1]$ is a minimal object in $\Bc_\omega$ by Proposition \ref{prop:Huy-minobjheart}, and so is necessarily stable.  In particular, $T_X[1]$ is a simple object, which implies  $\Omega_X[1]$  is also a simple object.  Since $\Omega_X[1]$ is an object in $\Bc_\omega$ with phase $1$ with respect to $\sigma_\omega$, it is $\sigma_\omega$-semistable and simple, and hence $\sigma_\omega$-stable.
\end{proof}

\section{Generalities on the cotangent bundle}

In this section, we prove some upper bounds of  the second Chern character $\ch_2$ of  subobjects of the cotangent bundle.  These bounds will be important in later sections in determining regions of stability for the cotangent bundle.

We begin with the following result due to Miyaoka:

\begin{lem}\cite[Lemma 10]{miyaoka1977chern}\label{lem:Miyaoka-subcotang-1}
Let $X$ be a smooth projective surface, and $L$ a line bundle generated by global sections.  If $\OO_X (D)$ is a subsheaf of $\Omega_X$, then either $Dc_1(L)\leq 0$ or $D^2 \leq 0$.
\end{lem}

Using \ref{lem:Miyaoka-subcotang-1}, we can prove the following lemma which will be useful for analyzing destabilizing sequences of the cotangent bundle:

\begin{lem}\label{lem:AG53-52-1}
Let $X$ be a smooth projective surface, $\omega$ an ample divisor on $X$, and $F$ a rank-one subsheaf of $\Omega_X$.  If $\omega c_1(F)>0$, then $c_1(F)^2 \leq 0$.
\end{lem}

\begin{proof}
Taking the double dual of $F$, we obtain  a short exact sequence of sheaves
\[
0 \to F \to F^{\ast\ast} \to T \to 0
\]
where $F^{\ast\ast}$ is a line bundle and $T$ is supported in dimension 0.  Since $\Ext^1(T,\Omega_X)=0$, applying $\Hom (-,\Omega_X)$ to this sequence gives $\Hom (F, \Omega_X) \cong \Hom (F^{\ast\ast}, \Omega_X)$, and so the inclusion $F \subset \Omega_X$ factors as
\[
\xymatrix{
  F \ar[r] \ar[dr] & F^{\ast\ast} \ar[d]^\psi \\
   & \Omega_X
}
\]
where the map $\psi$ is necessarily injective since $\psi$ is nonzero and $F^{\ast\ast}$ is rank-one torsion-free.  

Since $\omega$ is ample, $\OO_X (m\omega)$ is generated by global sections for $m \gg 0$.  Also, since $F^{\ast\ast}$ is a line bundle, it is of the form $\OO_X(D)$ for some divisor $D$ on $X$.  Note that   $\omega c_1(F)=\omega c_1(F^{\ast\ast})>0$ and so  $c_1(\OO_X(m\omega))c_1(F^{\ast\ast})>0$.  By  Lemma \ref{lem:Miyaoka-subcotang-1}, we have $c_1(F^{\ast\ast})^2 = c_1(F)^2 \leq 0$.
\end{proof}




\begin{lem}\label{lem:AG53-116-1}
Let $X$ be a smooth projective surface with ample canonical class $K$.  Suppose 
\[
0 \to A \to \Omega_X \to Q \to 0
\]
is a short exact sequence in both $\Bc_K$ and $\Coh (X)$, and that  $\ch_0(A)=2$.  Then  
\[
\ch_2(A) \leq \frac{(Kc_1(A))^2}{4K^2} \leq   \frac{K^2}{4}.
\]
\end{lem}

\begin{proof}
Suppose $A$ is not $\mu_K$-semistable.  Then it has a 2-step $\mu_K$-HN filtration 
\[
0 \neq F_0 \subsetneq F_1 = A
\]
where $F_0, F_1/F_0$ are both rank-one torsion-free sheaves.  Since $A$ lies in the torsion class $\Tc_K$ in $\Coh (X)$, we have $\mu_K(F_0)> \mu_K(F_1/F_0)>0$, i.e.\
\[
Kc_1(F_0) > Kc_1(F_1/F_0)>0.
\]
Since $Kc_1(F_0)+Kc_1(F_1/F_0) = Kc_1(A)>0$, it follows that 
\[
Kc_1(F_1/F_0) < \frac{Kc_1(A)}{2}
\]
and hence
\begin{align*}
\ch_2(F_1/F_0) &\leq \frac{c_1(F_1/F_0)^2}{2} \text{ by the Bogomolov-Gieseker inequality} \\
&\leq \frac{ (Kc_1(F_1/F_0))^2}{2K^2} \text{ by the Hodge index theorem} \\
&< \frac{ (Kc_1(A))^2}{8K^2}  
\end{align*}
On the other hand, $F_0$ is a rank-one subsheaf of $A$ and hence $\Omega_X$ that satisfies $Kc_1(F_0)>0$.  By Lemma \ref{lem:AG53-52-1}, it follows that $c_1(F_0)^2 \leq 0$ and so $\ch_2(A_0) \leq 0$ by the Bogomolov-Gieseker inequality.  Overall, we have
\[
\ch_2(A) = \ch_2(F_0) + \ch_2(F_1/F_0) \leq \frac{ (Kc_1(A))^2}{8K^2} 
\]
when $A$ is not $\mu_K$-semistable.


Now suppose $A$ is $\mu_K$-semistable.  The Bogomolov-Gieseker inequality and the Hodge index theorem together give
\[
  \ch_2(A) \leq \frac{ c_1(A)^2}{4} \leq \frac{ (Kc_1(A))^2}{4K^2}.
\]

Overall, we get 
\[
\ch_2(A) \leq  \frac{ (Kc_1(A))^2}{4K^2}
\]
regardless of whether $A$ is $\mu_K$-semistable.  Since $0 < Kc_1(A) \leq Kc_1(\Omega_X)=K^2$, the lemma follows.

\end{proof}

A result of Miyaoka's \cite[Remark 4.18]{miyaoka1984maximal}  says that on a smooth projective surface $X$, if $E$ is a rank-2 subsheaf of $\Omega_X^1 (\mathrm{log}\, D)$ whose determinantal divisor has a Zariski decomposition, then we have the inequality
\begin{equation}\label{eq:miyaokaineq2}
  3c_2(E) + \tfrac{1}{4}(\det{E})^2_- \geq c_1(E)^2
\end{equation}
where $(\det{E})_-$ denotes the negative part in the Zariski decomposition of  $\det{E}$ (see \cite[Section 3]{miyaoka1984maximal} for a definition). Here $(\det E)_-$ denotes the negative part $N$ in the Zariski decomposition $\det E=P+N$, where $P$ is nef and $N$ is an effective divisor whose irreducible components have a negative definite intersection matrix and are orthogonal to $P$.  We write down a special case of this inequality that will be useful later.

\begin{lem}\cite[Remark 4.18]{miyaoka1984maximal}\label{lem:Miyrmkcoro}
Let $X$ be a smooth projective surface with no negative curves.  Then for any rank-2 subsheaf $A$ of $\Omega_X$ where $c_1(A)$ is numerically effective, we have 
\[
\ch_2(A) \leq \frac{c_1(A)^2}{6}.
\]
\end{lem}

\begin{proof}
When $X$ is a smooth projective surface we can take $D=0$ so that $\Omega_X^1(\mathrm{log}\, D)=\Omega_X$.  If we further assume that $X$ has no negative curves, then $(\det{A})_{-}=0$ with respect to the Zariski decomposition, and \eqref{eq:miyaokaineq2} gives
\[
3c_2(A) \geq c_1(A)^2 \Leftrightarrow  \ch_2(A) \leq \frac{c_1(A)^2}{6}.
\]
\end{proof}

\section{Stability of the cotangent bundle on surfaces with ample canonical class - part 1}

When $X$ is a smooth projective surface with ample canonical class $K$ and we have a short exact sequence 
\[
0 \to A \to \Omega_X \to Q \to 0
\]
in the heart $\mathcal{B}_K$, \emph{a priori}, $A$ is not necessarily a subsheaf of $\Omega_X$ and its rank could well be higher than 2.  In this section, we prove criteria that ensure the  the rank of $A$ is at most 2, thus reducing the range of   subobjects $A$ that could potentially destabilize  $\Omega_X$ that we need to study.

\paragraph[Surfaces with ample canonical class]  We collect here some basic facts on surfaces with ample canonical class that will be relevant to us.

Let $X$ be a smooth projective surface.  We will write $K$ to denote its canonical divisor $K_X$.    Recall that $X$ is a \emph{minimal surface} if and only if  $K$ is nef \cite[Definition 2.13]{KM}, while $X$ is said to be a surface of \emph{general type} if $K$ is big.  As a result, whenever $K$ is ample, the surface $X$ itself is minimal of general type.  

The intersection numbers $K^2$ and $c_2(X) := c_2(T_X)$ play an important role in the classification of minimal surfaces of general type.  For any    minimal surface of general type $X$, both $K^2$ and $c_2(X)$ must be positive integers; moreover, the following inequalities must be satisfied (e.g.\ see \cite[(2.3)]{reider2021bridgeland}):
\begin{align}
    K^2 &\geq \frac{1}{5}(c_2(X)-36) \label{eq:msgt-ineq-noether} \\
    K^2 &\leq 3c_2(X). \label{eq:msgt-BMYineq}
\end{align}
The inequality \eqref{eq:msgt-BMYineq} is called the Bogomolov-Miyaoka-Yau (BMY) inequality.  Also, from Hirzebruch-Riemann-Roch \eqref{eq:HRR} we see that 
\begin{equation}\label{eq:mod12}
 K^2 + c_2(X) \equiv 0 \text{\quad mod }12.
\end{equation}

For a surface of general type $X$, we will follow the notation in \cite{reider2021bridgeland} and define the \emph{index} of the surface to be
\[
\tau_X :=  2\ch_2(T_X) = K^2-2c_2(X).
\]

\begin{rem}\label{rem:cotbdlss}
Additional properties of surfaces of general type include:
\begin{enumerate}
    \item[(i)] Given a surface of general type $X$, if  $p_g=0$ then $q=0$ \cite[(0.2)]{barlow1985simply} and hence $p_a=0$.
    \item[(ii)] Given a  smooth projective surface  $X$ where $K$ is ample, the cotangent bundle $\Omega_X$ is always $\mu_{K}$-semistable by a result of Aubin-Yau \cite[Theorem 3.1(2)]{peternell2001subsheaves}; if $\Omega_X$ is not a direct sum of line bundles, then it is $\mu_{K}$-stable \cite[Remark 3.2(2)(b)]{peternell2001subsheaves}.
    \item[(iii)] On a minimal surface of general type $X$, if $\tau_X>0$ then $\Omega_X$ is $\mu_{K}$-stable \cite[Lemma 2.3]{reider2021bridgeland}.
\end{enumerate}
\end{rem}

In latter sections, we will see that the constant
\begin{equation}\label{eq:def-constant-C}
  C_X := \frac{\tau_X}{2K^2} = \frac{1}{2} - \frac{c_2(X)}{K^2}
\end{equation}
also plays an important role in the stability of $\Omega_X$.  We work out some bounds for $C_X$ below.

\begin{lem}\label{lem:AG53-C-range}
Let $X$ be a minimal surface of general type with canonical class $K$, and let $C_X$ be as in \eqref{eq:def-constant-C}.   
\begin{itemize}
    \item[(i)] When $\tau_X >0$, we have $0<C_X\leq \tfrac{1}{6}$.
    \item[(ii)] When $\tau_X <0$ but $\tfrac{1}{5}c_2(X) \leq K^2$, we have $-\tfrac{9}{2} \leq C_X < 0$.
    \item[(iii)] When $K^2 < \tfrac{1}{5}c_2(X)$: 
      \begin{itemize}
          \item[(a)] If $c_2(X)>36$ then    $-\tfrac{369}{2} \leq C_X < -\tfrac{9}{2}$.
          \item[(b)] If $c_2(X)\leq 36$ then   $-\tfrac{71}{2} \leq C_X < -\tfrac{9}{2}$.
      \end{itemize}
\end{itemize}
\end{lem}

The different regions in this lemma can be visualized in Figure \ref{fig:CX-regions}.

\begin{figure}[htbp]
  \centering
  \includegraphics[width=0.9\textwidth]{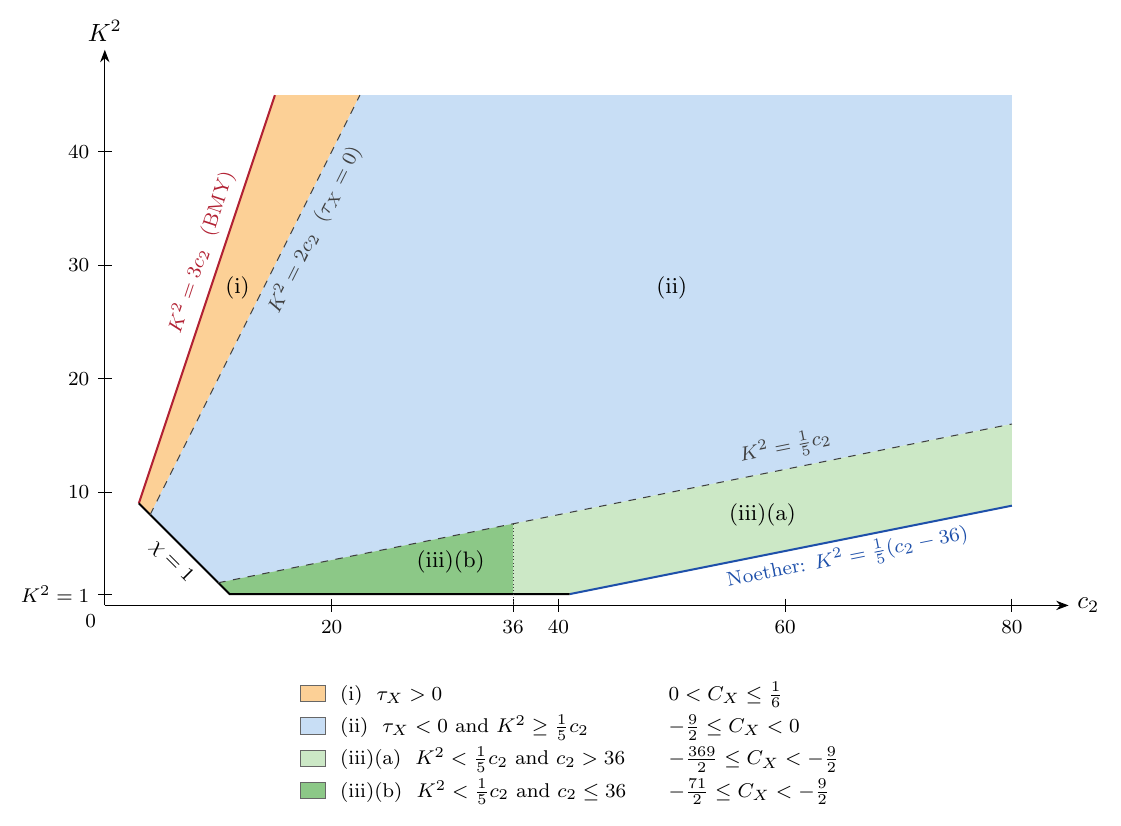}
  \caption{The admissible region for minimal surfaces of general type in the
  $(c_2(X),K^2)$-plane, bounded by the Bogomolov--Miyaoka--Yau inequality
  $K^2\le 3c_2$, Noether's inequality $K^2\ge\tfrac15(c_2-36)$, and, near the
  origin, by $\chi=1$ and $K^2=1$. The four shaded pieces correspond to the cases in
  Lemma \ref{lem:AG53-C-range}.  Note that the line $\tau_X=0$ is the line through the origin of slope 2, while the BMY line ($K^2=3c_2$) is the line through the origin of slope 3.}
  \label{fig:CX-regions}
\end{figure}


\begin{proof}
Suppose $\tfrac{1}{5}c_2(X) \leq K^2 \leq 3c_2(X)$.  Since $K^2 \leq 3c_2(X)$ is equivalent to $C_X \leq \tfrac{1}{6}$ and $\tfrac{1}{5}c_2(X) \leq K^2$ equivalent to $-\tfrac{9}{2} \leq C_X$, we obtain
\[
-\tfrac{9}{2} \leq C_X \leq \tfrac{1}{6}.
\]
Since $C_X$ and $\tau_X$ have the same sign, (i) and (ii) now follow.

To see (iii), suppose we are in the case 
\begin{equation}\label{eq:AG53-125-2}
 \tfrac{1}{5}(c_2(X)-36) \leq K^2 < \tfrac{1}{5}c_2(X),
 \end{equation}
or equivalently
 \[
\tfrac{1}{5}\left( 1 - \tfrac{36}{c_2(X)} \right)  \leq \tfrac{K^2}{c_2(X)} < \tfrac{1}{5}.
\]
We have two subcases:
\begin{itemize}
    \item If $c_2(X)>36$: then $1-\tfrac{36}{c_2(X)}$ is positive and we have 
\begin{equation}\label{eq:AG53-125-1}
\frac{1}{2} - \frac{1}{\tfrac{1}{5}\left( 1 - \tfrac{36}{c_2(X)}\right)} \leq \frac{1}{2} - \frac{c_2(X)}{K^2} \leq \frac{1}{2}-5.
\end{equation}
The left-hand side of the first inequality in \eqref{eq:AG53-125-1} achieves its minimum value when $c_2(X)$ achieves the minimum possible, which is $c_2(X)=37$,  so in this case we have 
\[
\frac{1}{2} - \frac{1}{\tfrac{1}{5}\cdot \tfrac{1}{37}} = -\frac{369}{2} \leq C_X \leq -\frac{9}{2}.
\]
  \item If $c_2(X)\leq 36$: then $\tfrac{1}{5}(c_2(X)-36) \leq 0$ and \eqref{eq:AG53-125-2} is equivalent to $1 \leq K^2 < \tfrac{1}{5}c_2(X)$, which in turn gives
  \[
  \tfrac{1}{2}-c_2(X) \leq \tfrac{1}{2} - \tfrac{c_2(X)}{K^2}=C_X < \tfrac{1}{2} - 5
  \]
  and hence $-\tfrac{71}{2} \leq C_X < -\frac{9}{2}$.
\end{itemize}
\end{proof}



To understand the Bridgeland stability of the cotangent bundle, we begin by considering what might cause the cotangent bundle to change from being stable to being unstable as we vary the ample class in the definition of a stability condition.  One way to ensure that the cotangent bundle is at least Bridgeland semistable at the start of this process is to assume it is twisted Gieseker semistable - a notion originally defined by Matsuki and Wentworth \cite{matsuki1994mumford}.

\begin{defn}[twisted Gieseker stability]\cite[Definition 4.1]{LM2}
On a smooth projective surface $X$ with ample class $\omega$ and $\mathbb{Q}$-divisor $B$, we say a torsion-free sheaf $E$ is \emph{$B$-twisted Gieseker semistable} if and only if 
\begin{enumerate}
    \item[(i)] $E$ is $\mu_\omega$-semistable, and
    \item[(ii)] for every subsheaf $A \subseteq E$ such that $\mu_\omega (A) = \mu_\omega (E)$, we have
    \[
      \frac{\chi(A \otimes L)}{\ch_0(A)} \leq \frac{\chi (E \otimes L)}{\ch_0(E)}
    \]
    where $L = -B + \tfrac{K_X}{2}$ and $\chi (- \otimes L)$ is computed using the Hirzebruch-Riemann-Roch formula.
\end{enumerate}
When $B=0$, we simply say $E$ is twisted Gieseker semistable.
\end{defn}

\paragraph \label{para:outermostwall-existence} Let $X$ be a smooth projective surface, $\omega$ an ample $\mathbb{Q}$-divisor on $X$, and $E$ a torsion-free, twisted Gieseker semistable sheaf with respect to $\omega$ and with $\omega \ch_1(E)>0$.     Then there exists $t_0>0$ such that $E$ is $\sigma_{t\omega}$-semistable for all $t \geq t_0$ by \cite[Lemma 4.8]{LM2}.  By the local finiteness of mini-walls \cite[Theorem 1.1]{LQ}, unless $E$ is  $\sigma_{t\omega}$-semistable for all $t>0$, there must be an `outermost mini-wall' $t_1 \in (0,t_0]$ satisfying the properties below.  
\begin{itemize}
    \item[(a)] $E$ is $\sigma_{t\omega}$-semistable for all $t \geq t_1$.
    \item[(b)] There exists a subobject $A$ of $E$ in $\Bc_\omega$ such that $\rho_{t_1\omega}(A)=\rho_{t_1\omega}(E)$ and $A$ is $\sigma_{t_1\omega}$-semistable, while  $\rho_{t\omega} (A) > \rho_{t\omega} (E)$ for all $0<t<t_1$.
\end{itemize}
Note that if we write $Q$ to denote  the cokernel  of the injection $A \to E$ in $\Bc_\omega$ at the outermost mini-wall, then  $\rho_{t_1\omega}(A)=\rho_{t_1\omega}(E)=\rho_{t_1\omega}(Q)$ and $Q$ is necessarily $\sigma_{t_1\omega}$-semistable.  We often refer to the short exact sequence in $\Bc_\omega$
\[
0\to A \to E \to Q \to 0
\]
as the \emph{destabilizing sequence} for $E$  at the outermost wall $t_1$.

Note that when $\omega$ is an ample divisor on a surface $X$, the categories $\mathcal{T}_{t\omega}, \mathcal{F}_{t\omega}$ and the heart $\mathcal{B}_{t\omega}$ are all  independent of $t>0$ and hence are equal to  $\mathcal{T}_\omega, \mathcal{F}_\omega, \mathcal{B}_\omega$ for all $t>0$, respectively.

Proposition \ref{prop:uppboundrkdestabsubobj} below shows that we can bound the rank of the subobject $A$.

\begin{prop}\label{prop:uppboundrkdestabsubobj}
Let $X$ be a smooth projective surface.  Suppose $\omega$ is an ample class on $X$, and $E$ is a torsion-free sheaf in $\mathcal{T}_\omega$.  Then for any short exact sequence in $\Bc_\omega$ of the form
\[
0 \to A \to E \to Q \to 0
\]
where $A$ is $\sigma_\omega$-semistable and $\rho_\omega (A) = \rho_\omega (E) = \rho_\omega (Q)$, we have
\begin{equation}\label{eq:uppboundrkdestabsubobj}
\ch_0(A) < \frac{\omega c_1(E)}{\omega^2 (D+\sqrt{D^2+1})}
\end{equation}
where 
\[
D = -\frac{\Re Z_\omega (E)}{\Im Z_\omega (E)} = \frac{\ch_2 (E) - \tfrac{\omega^2}{2}\ch_0(E)}{\omega \ch_1 (E)}.
\]
\end{prop}

\begin{proof}
From the associated long exact sequence, we know $A$ is a torsion-free sheaf. Let $F$ be the right-most $\mu_\omega$-HN factor of $A$.  Then the canonical surjection $A \to F$ is also a $\Bc_{\omega}$-surjection, and the $\sigma_{\omega}$-semistability of $A$ gives $\rho_{\omega}(A) \leq \rho_{\omega}(F)$.  Hence we have
\begin{align*}
\rho_{\omega}&(E) \leq \rho_{\omega}(F) \\
&\Leftrightarrow - \frac{\Re Z_\omega (E)}{\Im Z_\omega (E)}  \leq -\frac{\Re Z_\omega (F)}{\Im Z_\omega (F)} \\
&\Leftrightarrow  D \leq   \frac{ \ch_2(F) -\tfrac{\omega^2}{2}\ch_0(F)}{\omega \ch_1(F)}.
\end{align*}

Since $A$ has nonzero rank, the $\mu_\omega$-HN factor $F$ of $A$ also has nonzero rank.  We can then rewrite the above inequality as
\begin{equation}\label{eq:AG50-104-1b}
  D \frac{\omega \ch_1(F)}{\ch_0(F)} \leq \frac{\ch_2(F)}{\ch_0(F)} - \frac{\omega^2}{2}.
\end{equation}
Since $\rho_\omega (A) = \rho_\omega (E) = \rho_\omega (Q)$, we have $\Im Z_\omega = \omega \ch_1>0$ for both $E$ and $Q$, implying $0 < \omega \ch_1(A) < \omega c_1(E)$.  Note that we have $0 < \omega \ch_1(F) \leq \omega \ch_1(A)$, giving us 
$0 < \omega \ch_1(F) < \omega c_1(E)$.  We also have $0 < \ch_0(F) \leq \ch_0(A)$.  Now,
\[
  \ch_2(F) \leq \frac{\ch_1(F)^2}{2\ch_0(F)} \leq \frac{1}{2\ch_0(F)} \frac{ (\omega \ch_1(F))^2}{\omega^2}
\]
where the first inequality follows from the Bogomolov-Gieseker inequality for slope semistable sheaves, and the second inequality follows from the Hodge Index Theorem.  Dividing both sides by $\ch_0(F)$, we obtain
\[
  \frac{\ch_2(F)}{\ch_0(F)} \leq \frac{1}{2\omega^2} \mu_\omega (F)^2.
\]
Together with \eqref{eq:AG50-104-1b}, we now have
\[
\frac{1}{2\omega^2}\mu_\omega (F)^2 - D \mu_\omega (F)  - \frac{\omega^2}{2} \geq 0.
\]
The quadratic polynomial in $\mu_\omega (F)$ on the left-hand side has discriminant
\[
  D^2 +1 >0.
\]
Since $\mu_\omega (F)>0$, it follows that $\mu_\omega (F)$ is greater than the larger of the two roots of the above quadratic polynomial in $\mu_\omega (F)$, giving us
\[
  \mu_\omega (A) \geq \mu_\omega (F) > \omega^2 (D + \sqrt{D^2+1}).
\]
Hence
\[
  \ch_0(A) < \frac{\omega c_1(A)}{\omega^2 (D+\sqrt{D^2+1})} < \frac{\omega c_1(E)}{\omega^2 (D+\sqrt{D^2+1})}.
\]
\end{proof}

\begin{cor}\label{cor:AG53-70-1}
Let $X$ be a smooth projective surface with ample canonical class $K$, and suppose  $\omega = tK$ where  $t>0$. Let $A$ be a  $\sigma_\omega$-destabilizing subobject  of $\Omega_X$ in the sense of Proposition \ref{prop:uppboundrkdestabsubobj}. Then
\begin{itemize}
    \item[(i)] When $t^2 \geq \max{\left\{ -\tfrac{C_X}{2}, \tfrac{1-6C_X}{3}  \right\}}$, we have $\ch_0(A) \leq 2$.
    \item[(ii)] When $\tau_X \geq 0$, we have  $t^2 \geq \tfrac{1}{3}$ implies  $\ch_0(A) \leq 2$.
\end{itemize}
\end{cor}

Note that under the hypotheses of this corollary, even though $A$ has rank at most 2, it is not \emph{a priori} a subsheaf of $\Omega_X$ since $Q$ may be a genuine complex.  In the proof below, let us write $C$ for $C_X = \frac{\tau_X}{2K^2}$.

\begin{proof}
Using the notation from Proposition \ref{prop:uppboundrkdestabsubobj} and noting $\ch (\Omega_X) = (2, K, \tfrac{\tau_X}{2})$, we obtain
\[
D = \frac{\ch_2(\Omega_X) - \omega^2}{\omega K} = \left(  \frac{\tau_X}{2K^2}\right)\frac{1}{t}-t.
\]
Writing $s=1/t$, the bound on the right-hand side of \eqref{eq:uppboundrkdestabsubobj} (with $E=\Omega_X$) equals
\begin{align*}
    \frac{s}{D+ \sqrt{D^2+1}} &= s\frac{D-\sqrt{D^2+1}}{D^2-(D^2+1)} \\
    &= s (\sqrt{D^2+1}-D) \\
    &= \sqrt{(Cs^2-1)^2+s^2} - (Cs^2-1) 
\end{align*}
Now suppose  $t^2 \geq \max{\left\{ -\tfrac{C}{2}, \tfrac{1-6C}{3}  \right\}}$.  From above, we know
\[
\ch_0(A) <   \frac{s}{D+ \sqrt{D^2+1}}= \sqrt{(Cs^2-1)^2+s^2} - (Cs^2-1) 
\]
and so $\ch_0(A) \leq 2$ would follow if we can make sure 
\begin{align*}
    &\sqrt{(Cs^2-1)^2+s^2} - (Cs^2-1) \leq 3 \\
    \Leftrightarrow \,  &\sqrt{(Cs^2-1)^2+s^2}   \leq Cs^2 + 2  \\
    \Leftarrow \, &(Cs^2-1)^2+s^2 \leq (Cs^2+2)^2 \text{ since $t^2 \geq -\tfrac{C}{2} \Leftrightarrow Cs^2 +2 \geq 0$}\\
    \Leftrightarrow \, & (1-6C)s^2 \leq 3 \\
    \Leftrightarrow \, & t^2 \geq \tfrac{1-6C}{3}
\end{align*}
and so (i) follows.

Next, observe that $C=\frac{1}{2} - \frac{c_2(X)}{K^2} \leq \frac{1}{6}$ since $\frac{c_2(X)}{K^2} \geq \frac{1}{3}$ by the BMY inequality \eqref{eq:msgt-BMYineq}.  Also, $C$ and $\tau_X$ have the same sign.  Assuming $\tau_X \geq 0$, we  obtain  $C \in [0,\frac{1}{6}]$, in which case $\max{\left\{ -\tfrac{C}{2}, \tfrac{1-6C}{3}  \right\}}=\tfrac{1}{3}$. Part (ii) then follows from (i).
\end{proof}

The following lemma is a slight generalization of the last assertion in Corollary \ref{cor:AG53-70-1}.

\begin{lem}\label{lem:AG53-95-1}
Let $X$ be a smooth projective surface with ample canonical class $K$ such that 
\[
  \tfrac{6}{5}c_2(X) \leq K^2 \leq 3c_2(X).
\]
Then for $\omega = t K$ and any $t \geq 1$, every  $\sigma_\omega$-destabilizing subobject $A$ of $\Omega_X$ in the sense of  \ref{para:outermostwall-existence} has rank at most 2.  
\end{lem}

\begin{proof}
From the proof of  Corollary \ref{cor:AG53-70-1}, we know that
\[
\ch_0(A) <  \sqrt{(Cs^2-1)^2+s^2} - (Cs^2-1) 
\]
where $C =\frac{\tau_X}{2K^2}$.  Therefore, in order to obtain $\ch_0(A) \leq 2$, we just need
\begin{align*}
     &\text{\quad}\sqrt{(Cs^2-1)^2+s^2} - (Cs^2-1) \leq 3 \\
     &\Leftarrow \sqrt{(|C|+1)^2+1} + (|C|+1) \leq 3 \\
     &\Leftrightarrow \sqrt{(|C|+1)^2+1} \leq 2-|C| \\
     &\Leftrightarrow (|C|+1)^2 + 1 \leq (2-|C|)^2 \text{\quad provided $2-|C|\geq 0$} \\
     &\Leftrightarrow |\tau_X| \leq \tfrac{2}{3}K^2 \\
     &\Leftrightarrow \tfrac{6}{5}c_2(X) \leq K^2 \leq 6c_2(X).
\end{align*}
Since $K^2 \leq 3c_2(X)$ holds by the BMY inequality, and  $2-|C|\geq 0$ if and only if $K^2 \geq \tfrac{2}{5}c_2(X)$,  the claim follows.
\end{proof}

We now give a summary of how large $t$ needs to be in order to ensure the destabilizing subobject $A$ of $\Omega_X$ at the outermost mini-wall $\sigma_{tK}$ is of rank at most 2.

\begin{prop}\label{prop:bounds-t-cases}
Let $X$ be a smooth projective  surface with ample canonical class $K$.  Suppose $\omega = tK$ where $t \geq 1$, and $A$ is a $\sigma_\omega$-destabilizing subobject of $\Omega_X$ in the sense of  \ref{para:outermostwall-existence}.  Then
\begin{itemize}
    \item[(i)] When $\tfrac{6}{5}c_2(X) \leq K^2 \leq 3c_2(X)$,  we have $\ch_0(A) \leq 2$ whenever $t \geq 1$.
    \item[(ii)] When $C_X \leq 0$ (or equivalently $K^2\leq 2c_2(X)$), we have $\ch_0(A) \leq 2$ whenever $\tfrac{1}{t^2} < \min{\{ -\tfrac{2}{C_X}, \tfrac{3}{1-6C_X}\}}$.  In particular,
    \begin{itemize}
    \item[(a)] when $\tfrac{1}{5}c_2(X) \leq K^2 < \tfrac{6}{5}c_2(X)$, we have $\ch_0(A) \leq 2$ whenever $t^2>\tfrac{28}{3}$;
    \item[(b)] when $\tfrac{1}{5}(c_2(X)-36) \leq K^2 < \tfrac{1}{5} c_2(X)$ and $c_2(X) \neq 36$, we have $\ch_0(A) \leq 2$ whenever $t^2 > \tfrac{1108}{3}$.
    \item[(c)] when $\tfrac{1}{5}(c_2(X)-36) \leq K^2 < \tfrac{1}{5} c_2(X)$ and $c_2(X) = 36$, we have $\ch_0(A) \leq 2$ whenever $t^2 > \frac{214}{3}$.
    \end{itemize}
\end{itemize}
\end{prop}

\begin{proof}
(i): This is Lemma \ref{lem:AG53-95-1}.

(ii) Write $C$ for $C_X=\frac{\tau_X}{2K^2}$. Suppose $C \leq 0$, i.e.\ $K^2 \leq 2c_2(X)$.  Let us set
\[
f(C,s) = \sqrt{(Cs^2-1)^2+s^2}-(Cs^2-1)
\]
where  $s = \tfrac{1}{t} \in (0,1]$.  From the proof of  Corollary \ref{cor:AG53-70-1}, we know that 
\[
\ch_0(A) < f(C,s).
\]
Therefore, in order to force $\ch_0(A) \leq 2$, we just need $f(C,s)<3$, i.e.
\begin{equation}\label{eq:AG53-127-1}
    \sqrt{(Cs^2-1)^2+s^2} < Cs^2+2.
\end{equation}
Assuming $Cs^2+2 > 0$ (i.e.\ $s^2 < -\tfrac{2}{C}$), \eqref{eq:AG53-127-1} is equivalent to 
\begin{align*}
   \text{\quad} (Cs^2-1)^2+s^2 &< (Cs^2+2)^2 \\
    \Leftrightarrow   s^2(1-6C) &<3 \\
    \Leftrightarrow  s^2 &< \tfrac{3}{1-6C} \text{\quad since we are assuming $C \leq 0$}.
\end{align*} 
Overall we have $\ch_0(A) \leq 2$ whenever
\[
  s^2 < \min{\{ -\tfrac{2}{C}, \tfrac{3}{1-6C}\}}.
\]

(a) In this case, we have 
\[
-\tfrac{9}{5}c_2(X) \leq K^2 - 2c_2(X) =\tau_X < - \tfrac{4}{5}c_2(X)
\]
and subsequently 
\[
-\tfrac{9}{10}\tfrac{c_2(X)}{K^2} \leq C \leq -\tfrac{2}{5}\tfrac{c_2(X)}{K^2}.
\]
At the same time, the constraints $\tfrac{1}{5}c_2(X) \leq K^2 < \tfrac{6}{5}c_2(X)$ give
\[
\tfrac{5}{6} < \tfrac{c_2(X)}{K^2} \leq 5,  
\]
and so we have
\[
   -\tfrac{9}{2} \leq  C < -\tfrac{1}{3}.
\]
This means that $-\tfrac{2}{C} \geq \tfrac{4}{9}$ and $\tfrac{3}{1-6C} \geq \tfrac{3}{28}$, and so we have $\ch_0(A) \leq 2$ whenever $s^2 < \tfrac{3}{28}$, i.e.\ $t^2 > \tfrac{28}{3}$.

(b) In this case, we have 
\begin{equation}\label{eq:AG53-125-3}
\tfrac{1}{5}\left( 1 - \tfrac{36}{c_2(X)}\right) \leq \tfrac{K^2}{c_2(X)} < \tfrac{1}{5}.
\end{equation}

If $c_2(X)>36$, then 
\[
  \frac{1}{2}- \frac{1}{\tfrac{1}{5}\left( 1 - \tfrac{36}{c_2(X)} \right)}\leq \frac{1}{2} - \frac{c_2(X)}{K^2} = C \leq \frac{1}{2}-5.
\]
Since the left-most expression here is minimized when $c_2(X)=37$, we obtain
\[
\frac{1}{2}- \frac{1}{\tfrac{1}{5}\left( 1 - \tfrac{36}{37} \right)}  =-184.5 \leq C \leq -4.5.
\]

On the other hand, if $c_2(X)< 36$, then $\tfrac{1}{5}\left( 1 - \tfrac{36}{c_2(X)}\right)$ is negative and so the only information we have from \eqref{eq:AG53-125-3} is  $0 < \tfrac{K^2}{c_2(X)} < \tfrac{1}{5}$.  Since $K^2$ is a positive integer, it follows that 
\[
\tfrac{1}{c_2(X)} \leq \tfrac{K^2}{c_2(X)} < \tfrac{1}{5}
\]
and hence
\[
\tfrac{1}{2}-c_2(X) \leq \tfrac{1}{2} - \tfrac{c_2(X)}{K^2} < \tfrac{1}{2}-5, 
\]
giving us 
\[
-34.5 \leq C < -4.5.
\]

Overall, in case (b), we have $\ch_0(A) \leq 2$ whenever $s^2 <  \tfrac{3}{1+6(184.5)}$, i.e.\ when $t^2 > \tfrac{1108}{3}$.

(c) Under our assumptions in this case, we have $1 \leq K^2 \leq 7$ and hence
\[
C = \tfrac{1}{2}-\tfrac{c_2(X)}{K^2} \in \left\{ \tfrac{1}{2} - \tfrac{36}{m} : 1 \leq m \leq 7, m \in \mathbb{Z}^+  \right\}.
\]
As a result, 
\[
\min{\{ -\tfrac{2}{C}, \tfrac{3}{1-6C}\}} \leq \tfrac{3}{1-6\left( \tfrac{1}{2} - \tfrac{36}{7}\right)}=\tfrac{21}{202}
\]
and the claim follows.
\end{proof}

\bigskip

[For consistency, we should use $\rho_\omega$ rather than $\phi_t$ - change all - J]


From here on through the end of the article, given an object $Q \in D^b(X)$ on a smooth projective variety $X$, we will often write $Q^i$ to denote the cohomology $H^i(Q)$ of $Q$ with respect to the standard t-structure.


\begin{prop}\label{pro:AG54-24-1}
Let $X$ be a smooth projective surface with ample canonical class $K$.  Suppose 
\[
0 \to A \to \Omega_X \to Q \to 0
\]
is a short exact sequence in $\Bc_K$ where $A, Q \neq 0$.  If $\tau_X \geq 0$ and
\[
t^2 \geq \frac{\tfrac{5K^2}{8}}{\mu_K(\Omega_X)-\mu_K(A)}
\]
whenever $\ch_0(A) \geq 2$, then $\Omega_X$ is $\sigma_{tK}$-semistable.
\end{prop}

\begin{proof}
Consider the long exact sequence of sheaves
\[
0 \to Q^{-1} \to A \overset{\beta}{\to} \Omega_X \to Q^0 \to 0.
\]
Note that 
\begin{align}
    \rho_\omega(A) - \rho_\omega(\Omega_X) &= \frac{\ch_2(A)-\tfrac{t^2K^2}{2}\ch_0(A)}{tKc_1(A)} - \frac{\tfrac{\tau_X}{2}-t^2K^2}{tK^2} \notag \\
    &= \frac{1}{tKc_1(A)\cdot K^2}\left( t^2K^2\ch_0(A) (\mu_K(A)-\mu_K(\Omega_X) ) + K^2\ch_2(A) - \tfrac{\tau_X}{2}Kc_1(A)    \right) \label{eq:AG54-26-1}
\end{align}
where we note $\tfrac{K^2}{2}=\mu_K(\Omega_X)$.

We divide into two cases according to the rank of $\ch_0(A)$.

\noindent
\textbf{Case I: $\ch_0(A)=1$.} Since $\beta \neq 0$, $Q^{-1}$ must be 0 and so $A$ is a rank-one subsheaf of $\Omega_X$. By the Bogomolov-Gieseker inequality and Lemma \ref{lem:AG53-52-1}, we have $\ch_2(A) \leq \frac{c_1(A)^2}{2} \leq 0$.  Also, the $\mu_K$-semistability of $\Omega_X$ and the fact that $A$ lies in $\mathcal{T}_K$ give $0 < Kc_1(A)  \leq \frac{K^2}{2}$.  

Suppose $\tau_X\geq 0$.  
\begin{itemize}
    \item If $\mu_K(A) = \mu_K(\Omega_X)$, then  the right-hand side of \eqref{eq:AG54-26-1} is non-positive, meaning $\phi_t (A) \leq \phi_t(\Omega_X)$ for all $t>0$.
    \item If $\mu_K(A) < \mu_K(\Omega_X)$, then the right-hand side of \eqref{eq:AG54-26-1} is strictly negative, and so $\phi_t(A) < \phi_t (\Omega_X)$ for all $t>0$.
\end{itemize}

\noindent
\textbf{Case II: $\ch_0(A)\geq 2$.} In this case, $Q^{-1}$ does not necessarily vanish.  We will first find an upper bound for $\ch_2(A)$.  

Suppose the $\mu_K$-HN factors of $Q^{-1}$ are $J_1, \cdots, J_n$.  Then 
\begin{align*}
\ch_2(Q^{-1}) = \sum_{i=1}^n \ch_2(J_i) &\leq \sum_{i=1}^n \frac{c_1(J_i)^2}{2\ch_0(J_i)} \text{\quad by the Bogomolov-Gieseker inequality}\\
&\leq \sum_{i=1}^n \frac{(Kc_1(J_i))^2}{2K^2\ch_0(J_i)} \text{\quad by the Hodge index theorem}.
\end{align*}
Since each $J_i$ lies in $\Fc_K$ and hence has $Kc_1(J_i)\leq 0$, we have 
\[
0 \leq Kc_1(J_i[1]) \leq Kc_1(Q^{-1}[1]) \leq Kc_1(Q) < Kc_1(\Omega_X) =K^2
\]
for each $i$ and hence $(Kc_1(J_i))^2 < (K^2)^2$.  Therefore, the inequalities above give
\[
\ch_2(Q^{-1}) < \sum_{i=1}^n \frac{K^2}{2\ch_0(J_i)} \leq \frac{K^2}{2}\ch_0(Q^{-1}) \leq \frac{K^2}{2}\ch_0(A).
\]
On the other hand, 
\begin{itemize}
    \item If $\ch_0(\image \beta)=1$, then $\ch_2(\image \beta) \leq 0$ by the Bogomolov-Gieseker inequality and Lemma \ref{lem:AG53-52-1}.
    \item If $\ch_0(\image \beta)=2$, then by Lemma \ref{lem:AG53-116-1}, we have $\ch_2(\image \beta) \leq \tfrac{K^2}{4}$.
\end{itemize}
In either case, we have $\ch_0(\image \beta) \leq \tfrac{K^2}{4}$ and hence 
\begin{equation}\label{eq:ch2bound-01}
\ch_2(A) \leq \ch_2(Q^{-1})+\ch_2(\image \beta) < \frac{K^2}{2}\ch_0(A) + \frac{K^2}{4}.
\end{equation}

Next, we show that $\mu_K(A)< \mu_K(\Omega_X)$ unless $\dimension Q=0$.  To see this, note that since $Q^{-1} \in \Fc_K$, we have $0<Kc_1(A) \leq Kc_1(\image \beta)$.  We also have $\ch_0(A) \geq \ch_0(\image \beta)$.  It follows that 
\[
\mu_K(A) \leq \mu_K(\image \beta) \leq \mu_K(\Omega_X)
\]
where the second inequality follows from the $\mu_K$-semistability of $\Omega_X$.  Suppose $\mu_K(A)=\mu_K(\Omega_X)$.  Then we necessarily have $\mu_K(A)=\mu_K(\image \beta)$, and subsequently $\ch_0(A)=\ch_0(\image \beta)=2$ and $Kc_1(A)=Kc_1(\image \beta)$.  This implies $Q^{-1}=0$, and so $A$ is a rank-2 subsheaf of $\Omega_X$.  Then $Q=Q^0$ is a torsion sheaf.  If $\dimension Q=0$, then $\phi_t(A) < \phi_t(\Omega_X) < \phi_t(Q)=1$; if $\dimension Q=1$, then $Kc_1(A) < Kc_1(\Omega_X)$, contradicting $\mu_K(A)=\mu_K(\Omega_X)$.  Hence we cannot have $\mu_K(A)=\mu_K(\Omega_X)$.

Assuming $\tau_X \geq 0$, from \eqref{eq:AG54-26-1} and  \eqref{eq:ch2bound-01} we have
\begin{align*}
    \rho_\omega(A) - \rho_\omega(\Omega_X) &= \frac{1}{tKc_1(A)\cdot K^2}\left( t^2K^2\ch_0(A) (\mu_K(A)-\mu_K(\Omega_X) ) + K^2\ch_2(A) - \tfrac{\tau_X}{2}Kc_1(A)    \right) \\
    &\leq  \frac{1}{tKc_1(A)\cdot K^2} \left( t^2K^2\ch_0(A) (\mu_K(A)-\mu_K(\Omega_X) ) + \tfrac{(K^2)^2}{2}\ch_0(A)+\tfrac{(K^2)^2}{4}     \right) \\
    &=\frac{\ch_0(A)}{tKc_1(A)} \left( t^2 (\mu_K(A)-\mu_K(\Omega_X) ) + \tfrac{K^2}{2} + \tfrac{K^2}{4\ch_0(A)}   \right) \\
    &\leq \frac{1}{t\mu_K(A)} \left( t^2 (\mu_K(A)-\mu_K(\Omega_X) ) + \tfrac{5K^2}{8} \right) \text{\quad since $\ch_0(A) \geq 2$}.
\end{align*}

Overall, from our computations in Cases I and II, we see that under the assumptiosn $\tau_X \geq 0$ and $\ch_0(A) \geq 2$,  in order to ensure $\phi_t(A) \leq \phi_t(\Omega_X)$, it suffices to require
\[
 t^2 (\mu_K(A)-\mu_K(\Omega_X) ) + \tfrac{5K^2}{8} \leq 0,
\]
or equivalently
\[
t^2 \geq \frac{\tfrac{5K^2}{8}}{\mu_K(\Omega_X)-\mu_K(A)}
\]
for all subobjects $A$ of $\Omega_X$  in the heart $\Bc_K$ with $\ch_0(A) \geq 2$.
\end{proof}

Note that the first part of the proof of Proposition \ref{pro:AG54-24-1} gives the following lemma.

\begin{lem}\label{lem:AG54-24-1-der1}
Let $X$ be a smooth projective surface with ample canonical class $K$.  Suppose $0 \to A \to \Omega_X \to Q \to 0$ is a $\sigma_\omega$-destabilizing sequence in $\mathcal{B}_K$ in the sense of \ref{para:outermostwall-existence}.  If $\tau_X \geq 0$, then $\ch_0(A) \geq 2$.
\end{lem}

\begin{lem}\label{lem:AG54-29-1}
Let $X$ be a smooth projective surface with ample canonical class $K$ and Picard rank 1.   Suppose $0 \to A \overset{\beta}{\to} \Omega_X \to Q \to 0$ is a  short exact sequence in $\Bc_K$ with $A, Q\neq 0$.  Then the image of $\beta$ as a morphism in $\Coh (X)$ is of rank 2.
\end{lem}

\begin{proof}
Since $\image \beta$ is a quotient sheaf of $A$, it lies in the torsion class $\Tc_K$ and so $Kc_1(\image \beta) >0$.  If $\image \beta$ is of rank 1, however, Lemma \ref{lem:AG53-52-1} asserts that $c_1(\image \beta)^2 \leq 0$, which is impossible since $X$ has Picard rank 1.
\end{proof}

We end this section with a quick result saying that when $K^2$ attains  the  smallest possible value for $Kc_1$, the cotangent bundle   $\Omega$ is $\sigma_{tK}$-semistable for all $t>0$.

\paragraph \label{para:defn-cm} For  a smooth projective surface of general type $X$ with canonical class $K=K_X$, let us set
\[
c := \min{ \{ Kc_1(E) >0 : E \in \Coh (X) \}}
\]
and then let $m$ be the positive integer satisfying $K^2=mc$.

\begin{lem}\label{lem:surf-m-is-1}
Let $X$ be a smooth projective surface with ample canonical class $K$.  If $m=1$, then $\Omega_X$ is $\sigma_{tK}$-semistable for all $t>0$.
\end{lem}

Note that all the minimal surfaces of general type satisfying $K^2_X=1$, e.g.\  the numerical Godeaux surfaces \cite[Section 2.1]{bauer2011surfaces} (which also satisfy $p_g=q=0$) satisfy $m=1$.

\begin{proof}
Suppose there is a destabilizing sequence
\[
0 \to A \to \Omega_X \to Q \to 0
\]
in $\Bc_K$ corresponding to the outer-most wall in the family $\{\sigma_{tK}\}_{t>0}$ as in \ref{para:outermostwall-existence}.  Then  $A, Q \neq 0$ and $\phi (A)=\phi (\Omega_X)=\phi (Q)$.  Note that if $Kc_1(Q)=0$, then $\phi (Q)=1$ and we have $\phi (\Omega_X) < \phi (Q)$, a contradiction.  Hence $Kc_1(Q)=c$ and $Kc_1(A)=0$.  However, $A$ is a coherent sheaf in $\Bc_K$ and so must lie in $\Tc_K$, in which case $A$ must be a torsion sheaf supported in dimension 0, which is impossible since $A$ is a nonzero subobject of $\Omega_X$ in $\Bc_K$.   Hence there is no destabilizing sequence of the form above for any $t>0$.  

Since $Kc_1(\Omega_X)=c$, the cotangent bundle $\Omega_X$ cannot split into a direct sum of two line bundles with the same $\mu_K$-slope.  This means that $\Omega_X$ is $\mu_K$-stable, and hence twisted Gieseker semistable.  Then by \cite[Lemma 4.8]{LM2}, $\Omega_X$ is $\sigma_{tK}$-semistable for $t \gg 0$.  Combined with the result of the previous paragraph, it follows that $\Omega_X$ is $\sigma_{tK}$-semistable for all $t>0$.
\end{proof}

\section{Stability of the cotangent bundle on surfaces with nonnegative index}

This goal of this section is to compute a general bound for the stability of the cotangent bundle, on a smooth projective surface with ample canonical class and \emph{nonnegative} index $\tau_X$.  As we saw in Corollary \ref{cor:AG53-70-1}, when the index is nonnegative, the destabilizing subobject $A$ of the cotangent bundle $\Omega_X$ has rank at most 2, thus simplifying our analysis of the stability of $\Omega_X$.


Throughout this section, suppose $X$ is a smooth projective surface with ample canonical class $K$.  Write $\omega = tK$ where $t>0$.  

\paragraph \label{para:t-generalbound} Suppose $0 \to A \to \Omega_X \to Q \to 0$ is a short exact sequence in $\Bc_K$ with $A, Q \neq 0$ and 
\[
0 \to Q^{-1} \to A \overset{\eta}{\to} \Omega_X \to Q^0 \to 0
\]
is the associated long exact sequence in $\Coh (X)$, with $Q^i = H^i(Q)$.  Then we have the equivalences
\begin{align}
    \rho_\omega (A) > \rho_\omega (\Omega_X) &\Leftrightarrow  \frac{\ch_2(A)-\tfrac{\omega^2}{2}\ch_0(A)}{\omega \ch_1(A)} >  \frac{\ch_2(\Omega_X)-\tfrac{\omega^2}{2}\ch_0(\Omega_X)}{\omega \ch_1(\Omega_X)} \notag \\
    &\Leftrightarrow \frac{\ch_2(A) - \tfrac{K^2}{2}t^2\ch_0(A)}{Kc_1(A)} > \frac{\ch_2(\Omega_X) - \tfrac{K^2}{2}t^2\ch_0(\Omega_X)}{Kc_1(\Omega_X)} \notag \\
    &\Leftrightarrow \frac{K^2}{2} \left( \frac{\ch_0(A)}{Kc_1(A)} - \frac{\ch_0(\Omega)}{Kc_1(\Omega_X)} \right)t^2 < \frac{\ch_2(A)}{Kc_1(A)} - \frac{\ch_2(\Omega_X)}{Kc_1(\Omega_X)} \notag \\
    &\Leftrightarrow \frac{K^2}{2} \left( \frac{1}{\mu_K(A)} - \frac{1}{\mu_K(\Omega_X)} \right)t^2 < \frac{\ch_2(A)}{Kc_1(A)} - \frac{\ch_2(\Omega_X)}{Kc_1(\Omega_X)}    \label{eq:t-generalbound-0}
\end{align}
Also note that
\begin{align*}
\mu_K(A) = \frac{Kc_1(A)}{\ch_0(A)} &= \frac{Kc_1(\image \eta)+Kc_1(Q^{-1})}{\ch_0(\image \eta) + \ch_0 (Q^{-1})} \\
&\leq \mu_K(\image \eta) \text{\quad since $Kc_1(Q^{-1})\leq 0$ while $Kc_1(\image \eta) > 0$} \\
&\leq \mu_K (\Omega_X)     \text{\quad since $\Omega_X$ is $\mu_K$-semistable}
\end{align*}    
If $\mu_K(A)=\mu_K(\Omega_X)$, it would force $\mu_K(A)=\mu_K(\image \eta)=\mu_K (\Omega_X)$ as well as  $\ch_0(Q^{-1})=0$ (or else $\mu_K(A)<\mu_K(\image \eta)$), meaning $Q^{-1}=0$ and $A$ is a subsheaf of $\Omega_X$.  We can then divide into two cases:
\begin{itemize}
    \item[(a)] If $\ch_0(A)=\ch_0(\Omega_X)=2$, then it follows that $Q$ is a 0-dimensional sheaf, and by the seesaw property we have $\rho_\omega (A) < \rho_\omega (\Omega_X) < \rho_\omega (Q)$.
    \item[(b)] If $\ch_0(A)=1$, then  $Kc_1(A) = \mu_K(A) = \mu_K(\Omega_X)=\tfrac{K^2}{2}$; on the other hand,     by Lemma \ref{lem:AG53-52-1} we have  $\ch_2(A) \leq 0$.  Hence  
    \[
        \rho_\omega (A) \leq \frac{-\tfrac{K^2}{2}t}{Kc_1(A)} = -t.
    \]
    On the other hand,
    \[
    \rho_\omega (\Omega_X) = \frac{\tfrac{\tau_X}{2} - K^2t^2}{tK^2}=\tfrac{1}{t}C-t.
    \]
    In this case, if we additionally assume $\tau_X \geq 0$ then $C \geq 0$ and we get $\rho_\omega (A) \leq \rho_\omega (\Omega_X)$ for all $t>0$.  Also, if we assume $\Omega_X$ is $\mu_K$-stable instead of $\tau_X \geq 0$, then we must have $\mu_K(A) < \mu_K(\Omega_X)$ and this case cannot occur.  If $K^2$ is an odd integer, then clearly this case cannot occur, either, since $Kc_1(A)$ is an integer.
\end{itemize}
The  computation  in \ref{para:t-generalbound} now gives   the following proposition.

\begin{lem}
Let $X$ be a smooth projective surface with ample canonical class $K$.  For any short exact sequence $0 \to A \to \Omega_X \to Q \to 0$ in $\Bc_K$ where  $\rho_\omega (A) > \rho_\omega (\Omega_X)$, if one of the conditions 
\begin{itemize}
    \item[(a)] $\tau_X \geq 0$
    \item[(b)] $\Omega_X$ is $\mu_K$-stable
    \item[(c)] $K^2$ is an odd integer
\end{itemize}
holds then  $\mu_K(A) < \mu_K (\Omega_X)$ and \eqref{eq:t-generalbound-0} can be rewritten as \begin{equation}\label{eq:t-generalbound-1}
    t^2 < \frac{\frac{\ch_2(A)}{Kc_1(A)} - \frac{\ch_2(\Omega_X)}{Kc_1(\Omega_X)}}{ \frac{\ch_0(A)}{Kc_1(A)} - \frac{\ch_0(\Omega)}{Kc_1(\Omega_X)}}\cdot \frac{2}{K^2}
\end{equation}
\end{lem}

Note that by  Remark \ref{rem:cotbdlss}, when $\tau_X >0$ the cotangent bundle $\Omega_X$ is necessarily $\mu_K$-stable.





\paragraph \label{para:t-generalbound-2} Again, take any short exact sequence  $0 \to A \to \Omega_X \to Q \to 0$     in $\Bc_K$ with $A, Q \neq 0$.  From \ref{para:t-generalbound}, we know that $\mu_K(A) \leq \mu_K(\Omega_X)$.  If $\mu_K(A) = \mu_K(\Omega_X)$, then 
\[
\rho_\omega (A) - \rho_\omega (\Omega_X) = \frac{\ch_2(A)}{Kc_1(A)} - \frac{\ch_2(\Omega_X)}{Kc_1(\Omega_X)}
\]
which is independent of $t$.    In particular,  this cannot happen if $0 \to A \to \Omega_X \to Q \to 0$   is a destabilizing sequence with respect to $\sigma_\omega$ in the sense of \ref{para:outermostwall-existence}.  In other words, whenever we have a destabilizing sequence  with respect to $\sigma_\omega$ in the sense of \ref{para:outermostwall-existence}, we can assume $\mu_K(A) < \mu_K(\Omega_X)$.

The following lemma will be useful in the proof of the next theorem.

\begin{lem}\label{lem:Ach2bdrk2}
Let $X$ be a smooth projective surface with ample canonical class $K$.   Given a short exact sequence $0 \to A \to \Omega_X \to Q \to 0$ in $\mathcal{B}_K$ where $\ch_0(A) = 2$, we have 
\[
\ch_2(A) \leq \frac{ (Kc_1(A))^2}{2K^2}.
\]
\end{lem}

\begin{proof}
If $A$ is $\mu_K$-semistable, then the Bogomolov-Gieseker inequality and the Hodge Index Theorem together give
\[
\ch_2(A) \leq \frac{c_1(A)^2}{4} \leq \frac{ (Kc_1(A))^2}{4K^2}.
\]
If $A$ is not $\mu_K$-semistable, then it has two rank-one $\mu_K$-HN factors $A_0, A_1$ such that 
\[
Kc_1(A_0) > Kc_1(A_1) >0
\]
while
\[
Kc_1(A_0) + Kc_1(A_1) = Kc_1(A).
\]
Now 
\begin{align*}
    \ch_2(A) &= \ch_2(A_0) + \ch_2(A_1) \\
    &\leq \tfrac{ c_1(A_0)^2}{2} +  \tfrac{ c_1(A_1)^2}{2}  \\
    &\leq \tfrac{ (Kc_1(A_0))^2}{2K^2} + \tfrac{ (Kc_1(A_1))^2}{2K^2}  \\
    &\leq \tfrac{ (Kc_1(A_0) + Kc_1(A_1))^2}{2K^2} \text{ since $Kc_1(A_0), Kc_1(A_1)$ have the same sign} \\
    &= \tfrac{ (Kc_1(A))^2}{2K^2}
\end{align*}

Therefore, we have $\ch_2(A) \leq  \tfrac{ (Kc_1(A))^2}{2K^2}$ overall regardless of whether $A$ is $\mu_K$-semistable.
\end{proof}

\begin{thm}\label{thm:taunn-mainboundv2}
Suppose $X$ is a smooth projective surface with ample canonical class $K$ such that $\tau_X \geq 0$ and $\Omega_X$ is twisted $K$-Gieseker semistable.  
Then the cotangent bundle $\Omega_X$ is $\sigma_\omega$-semistable if
    \[
 t^2 >   \tfrac{K^2-1}{4K^2} \left( 4c_2(X) - K^2-1\right).
\]
In particular, 
 $\Omega_X$ is $\sigma_\omega$-semistable if 
\[
 t^2 >   \tfrac{K^2}{4}.
\]
\end{thm}

\begin{proof}
Suppose 
\[
0 \to A \to \Omega_X \to Q \to 0
\]
is a destabilizing sequence in $\mathcal{B}_K$ with respect to $\sigma_\omega$ in the sense of \ref{para:outermostwall-existence}.   Let us assume from here on that $t^2 \geq \tfrac{1}{3}$, so that Corollary  \ref{cor:AG53-70-1}(ii) and Lemma \ref{lem:AG54-24-1-der1} together give $\ch_0(A)=2$.  Consider the long exact sequence of coherent sheaves
\[
0 \to Q^{-1} \to A \overset{\eta}{\to} \Omega_X \to Q^0 \to 0
\]
where $Q^i$ denotes $H^i(Q)$.  By \ref{para:t-generalbound-2}, we can assume $\mu_K(A) < \mu_K (\Omega_X)$.  The equation $\rho_\omega (A) = \rho_\omega (\Omega_X)$ can now be rewritten by replacing all the inequalities in \eqref{eq:t-generalbound-0} by equalities, eventually giving us equality in \eqref{eq:t-generalbound-1}.

We now divide into two cases according to the value of $Kc_1(A)$:


\noindent
\textbf{Case I: $Kc_1(A) > \tfrac{K^2}{2}$.}  Since $Kc_1(Q^{-1}) \leq 0$, we must have $Kc_1(\image \eta) > \tfrac{K^2}{2}$.  Since $\Omega_X$ is $\mu_K$-semistable, it follows that the rank of $\image \eta$ cannot be 1 and must be 2.  This then forces $Q^{-1}$ to have rank 0 and therefore vanish.  Hence $A$ is a rank-2 subsheaf of $\Omega_X$.  Lemma \ref{lem:AG53-116-1} now applies and gives us $\ch_2 (A) \leq \frac{ (Kc_1(A))^2}{4K^2}$. 

Writing $x$ to denote $Kc_1(A)$, the right-hand side of \eqref{eq:t-generalbound-1} is now bounded from above by  
\[
b_1 (x) :=  \frac{\frac{x}{4K^2} - \frac{\ch_2(\Omega_X)}{Kc_1(\Omega_X)}}{ \frac{2}{x} - \frac{\ch_0(\Omega)}{Kc_1(\Omega_X)}}\cdot \frac{2}{K^2}.
\]
Note that  $\frac{d}{dx} b_1(x)$ has the same sign as
\[
\frac{1}{4K^2}\left( \frac{2}{x} - \frac{\ch_0(\Omega)}{Kc_1(\Omega_X)} \right) - \left( \frac{x}{4K^2} - \frac{\ch_2(\Omega_X)}{Kc_1(\Omega_X)} \right) \left( \frac{-2}{x^2} \right);
\]
multiplying this by $x^2K^2$ gives
\[
d_1(x) :=  \tfrac{-1}{2K^2}x^2 + x - \tau_X
\]
which has the same sign as $b_1'(x)$.  The discriminant of $d_1(X)$ is
\[
\Delta_1 (x) := -1 + 4\tfrac{c_2(X)}{K^2}
\]
Since $\tau_X \geq 0$, we have $2c_2(X) \leq K^2 \leq 3c_2(X)$ and hence $\tfrac{1}{3} \leq \tfrac{c_2(X)}{K^2} \leq \tfrac{1}{2}$.  Hence $\Delta_1(x)>0$ and the quadratic polynomial $d_1(x)$ has two real roots in $x$.

The larger root of $d_1(x)$ is
\[
  K^2 (1 + \sqrt{\Delta_1}) 
\]
which is strictly larger than $K^2$.  The smaller root of $d_1(x)$ is 
\[
  K^2(1-\sqrt{\Delta_1}) \leq K^2 \left(1-\sqrt{\tfrac{1}{3}}\right) < \tfrac{1}{2}K^2.
\]
Since $\rho_\omega (Q)< \infty$, we have $Kc_1(Q)>0$ and so $x=Kc_1(A)$ is an integer in the interval $\left( \tfrac{1}{2}K^2, K^2\right)$.  (In particular, we must have  $K^2 \geq 3$.)  The calculation above shows that over this interval,  $d_1(x)$ is always positive and 
\begin{align*}
    b_1(x) &\leq b_1(K^2-1) \\
    &=  \frac{\frac{K^2-1}{4K^2} - \frac{\tau_X/2}{K^2}}{ \frac{2}{K^2-1} - \frac{2}{K^2}}\cdot \frac{2}{K^2} \\
    &= \tfrac{1}{4}\cdot \tfrac{K^2-1}{K^2} \left( 4c_2(X) - K^2-1\right).
\end{align*}

Note that in the string of equivalences in \eqref{eq:t-generalbound-0}, if any of the inequalities achieves equality then so do the others.  Hence whenever $t^2 > b_1(K^2-1)$, the stability $\sigma_{tK}$ must be above the outer-most mini-wall and $\Omega_X$ must be $\sigma_\omega$-semistable.

In particular, since  $\tfrac{1}{3}K^2 \leq c_2(X) \leq \tfrac{1}{2}K^2$ from the assumption $\tau_X \geq 0$ and the BMY inequality,  we have
\begin{align*}
    b_1(K^2-1) &< \tfrac{1}{4} \cdot  \tfrac{K^2-1}{K^2}  ( K^2-1 ) < \tfrac{K^2}{4}, \text{ and} \\
    b_1(K^2-1) &\geq \tfrac{1}{4}\cdot \tfrac{K^2-1}{K^2}\left( \tfrac{1}{3}K^2 -1 \right) \geq 0
\end{align*}
which implies
\begin{equation}\label{eq:b1bounds}
0 \leq b_1(K^2-1) =  \tfrac{1}{4}\cdot \tfrac{K^2-1}{K^2} \left( 4c_2(X) - K^2-1\right) < \tfrac{K^2}{4}.
\end{equation}

\noindent
\textbf{Case II: $Kc_1(A) \leq \tfrac{K^2}{2}$.}  In this case, we do not necessarily  know that $A$ is a rank-2 subsheaf of $\Omega_X$.  We have $\ch_2 (A) \leq \tfrac{(Kc_1(A))^2}{2K^2}$ from Lemma \ref{lem:Ach2bdrk2}, and so  the right-hand side of \eqref{eq:t-generalbound-1} is  bounded from above by  
\[
b_2 (x) :=  \frac{\frac{x}{2K^2} - \frac{\ch_2(\Omega_X)}{Kc_1(\Omega_X)}}{ \frac{2}{x} - \frac{\ch_0(\Omega)}{Kc_1(\Omega_X)}}\cdot \frac{2}{K^2}
\]
where $x$ again denotes $Kc_1(A)$.  Similar to Case I, the derivative $\tfrac{d}{dx} b_2(x)$ has the same sign as 
\[
d_2 (x) := -\tfrac{1}{K^2}x^2 + 2x - \tau_X
\]
which has discriminant
\[
\Delta_2 (x) := 4-\tfrac{4\tau_X}{K^2} = 8\tfrac{c_2(X)}{K^2}.
\]
Note that the larger root of $d_2(x)$ is 
\[
\tfrac{K^2}{2}\left( 2 + \sqrt{\Delta_2}\right)
\]
which is larger than $K^2$.  Since $\tfrac{1}{3} \leq \tfrac{c_2(X)}{K^2} \leq \tfrac{1}{2}$,  the smaller root of $d_2(x)$ is
\[
\tfrac{K^2}{2}\left( 2 - \sqrt{\Delta_2}\right) \leq \tfrac{K^2}{2} \left( 2 -\sqrt{\tfrac{8}{3}}\right) < \tfrac{K^2}{2}.
\]
Note also that
\[
\tfrac{K^2}{2}\left( 2 - \sqrt{\Delta_2}\right) \geq \tfrac{K^2}{2} \left( 2 -\sqrt{4}\right) \geq 0
\]
Hence both roots of $d_2(x)$ are nonnegative, and over the interval $\left( 0, \tfrac{K^2}{2} \right]$ the value of $b_2(x)$ is bounded from above by the maximum between $b_2(1)$ and $b_2\left( \tfrac{K^2}{2}\right)$.  Since
\begin{align*}
    b_2\left( \tfrac{K^2}{2}\right) &=  \frac{\frac{K^2/2}{2K^2} - \frac{\ch_2(\Omega_X)}{Kc_1(\Omega_X)}}{ \frac{2}{K^2/2} - \frac{\ch_0(\Omega)}{Kc_1(\Omega_X)}}\cdot \frac{2}{K^2} = \tfrac{1}{4} - \tfrac{\tau_X}{2K^2} = -\tfrac{1}{4} + \tfrac{c_2(X)}{K^2} \leq \tfrac{1}{4}, \text{ and}\\
    b_2(1) &= \frac{\frac{1}{2K^2} - \frac{\ch_2(\Omega_X)}{Kc_1(\Omega_X)}}{ 2 - \frac{\ch_0(\Omega)}{Kc_1(\Omega_X)}}\cdot \frac{2}{K^2} = \tfrac{1}{2}\cdot \tfrac{1-\tau_X}{K^2(K^2-1)} = \tfrac{1}{2} \left( \tfrac{-1}{K^2} + \tfrac{2c_2(X)}{K^2(K^2-1)} \right) < \tfrac{c_2(X)}{K^2(K^2-1)} \leq \tfrac{1}{2}
\end{align*}
(note that $x = Kc_1(A)$ is necessarily $\geq 2$ in Case II), we get that $b_2(x) \leq \tfrac{1}{2}$ overall in Case II.  A similar argument to Case I then shows, that when we are in Case II, whenever $t^2 > \tfrac{1}{2}$ we have that $\Omega_X$ is $\sigma_\omega$-semistable.

Putting Case I and Case II together, we see that  $\Omega_X$ is $\sigma_\omega$-semistable whenever 
\begin{equation}\label{eq:nonnegindexgenbounds}
t^2 > \max{ \left\{ \tfrac{1}{3}, b_1(K^2-1), b_2\left( \tfrac{K^2}{2}\right), b_2(1)  \right\} }.
\end{equation}
In particular,  $\Omega_X$ is $\sigma_\omega$-semistable whenever 
\[
 t^2 > \max{ \left\{ \tfrac{1}{2},   \tfrac{K^2-1}{4K^2} \left( 4c_2(X) - K^2-1\right) \right\} }.
\]
Since we already know $\Omega_X$ is $\sigma_\omega$-stable for all $t>0$ when $K^2=1$ from Lemma \ref{lem:surf-m-is-1}, we will assume $K^2\geq 2$ from now on.

Now, we compare the magnitude of the term  $\tfrac{K^2-1}{4K^2} \left( 4c_2(X) - K^2-1\right)$ with $\tfrac{1}{2}$.  In the  calculations below, let us write for the sake of convenience
\[
 x = c_2(X), \text{\quad} y = K^2.
\]
Then 
\begin{align*}
    \tfrac{y-1}{4y}(4x-y-1) &> \tfrac{1}{2} \\
    \Leftrightarrow  y^2 + (2-4x)y + (4x-1) &<0
\end{align*}
Let us write
\[
f_x(y) :=  y^2 + (2-4x)y + (4x-1).
\]
Then the discriminant of $f_x(y)$ as a quadratic polynomial in $x$ is 
\[
\Delta_x = 8(2x^2-4x+1) 
\]
which is positive for all $x \geq 2$.  Since there are no minimal surfaces of general type with $x=c_2(X)=1$, we can assume $x \geq 2$ from now on.  

The roots of $f_x(y)$ as a quadratic polynomial in $y$ are
\[
\tfrac{1}{2}( (4x-2) \pm \sqrt{\Delta_x}) = (2x-1) \pm \sqrt{2(2x^2-4x+1)}.
\]
Between the two roots,
\begin{itemize}
    \item the larger root is strictly larger than $3x$ (equivalently, the point $(x,y)$ is above the BMY line $y=3x$) iff
    \begin{align*}
        \sqrt{2(2x^2-4x+1)} &> x+1 \\
        \Leftrightarrow 3x^2-10x+1 &>0.
    \end{align*}
    Let us write $g(x) := 3x^2-10x+1$.  Note that  $g(x)<0$ for $x=2,3$ and $g(x)>0$ for $x \geq 4$.  

    We cannot have $x=2$, for then the BMY inequality gives $y\leq 6$, contradicting \eqref{eq:mod12}.    
    When $x \geq 4$, we have $g(x)>0$ meaning the point $(x,y)$ is  above the BMY line on the $xy$-plane, and so such a surface $X$ does not exist.  When $x=3$, the BMY inequality together with \eqref{eq:mod12} implies $y=9$, i.e.\ $(x,y)=(3,9)$; we will come back to this case later.
    \item the smaller root is strictly less than $2$ iff
    \begin{align*}
        (2x-1) - \sqrt{2(2x^2-4x+1)} &< 2 \\
        \Leftrightarrow 4x+11 &> 0
    \end{align*}
    which is true for all $x \geq 2$.
\end{itemize}
Overall, we see that when $K^2 \geq 2$ and $(c_2(X),K^2)\neq (3,9)$, we always have $f_x(y) <0$, and so $\Omega_X$ is  $\sigma_\omega$-semistable whenever 
\[
t^2 >  \tfrac{K^2-1}{4K^2} \left( 4c_2(X) - K^2-1\right).
\]

In the case $(x,y)=(3,9)$, we have
\[
  b_1(K^2-1)=\tfrac{4}{9}, \text{\quad} b_2\left( \tfrac{K^2}{2}\right) = \tfrac{1}{12}, \text{\quad}  b_2(1) = -\tfrac{1}{72}.
\]
and so the maximum value on the right-hand side of \eqref{eq:nonnegindexgenbounds} is still attained by $b_1(K^2-1)$.  Hence we also have the $\sigma_\omega$-semistability of $\Omega_X$ whenever $t^2> b_1(K^2-1)=\tfrac{K^2-1}{4K^2} \left( 4c_2(X) - K^2-1\right)$ in this case.

Lastly, from \eqref{eq:b1bounds}, we also get that 
 $\Omega_X$ is  $\sigma_\omega$-semistable whenever 
\[
t^2 >  \tfrac{K^2}{4}.
\]
\end{proof}



\section{Stability of the cotangent bundle on surfaces with ample canonical class}

In this section, our goal is to  write down  a bound $t_0>0$ for the stability of the cotangent bundle $\Omega_X$ on an \emph{arbitrary} smooth projective surface $X$ with an ample canonical class $K$, so that for any $t>t_0$, the cotangent bundle is semistable with respect to the stability condition $\sigma_{tK}$.  This is consistent with  the intuition coming from differential geometry that was mentioned in Section \ref{sec:intro}, namely, since the LYZ equation for $\Omega_X$ is not scaling invariant, there should be a threshold $t_0>0$ such that, for $0<t\leq t_0$, the nonlinear effects of the quadratic curvature term of the LYZ equation begins to obstruct the existence of a solution with respect to the K\"{a}hler form $t\omega_{KE}$.




\textbf{Notation.}  Throughout this section, whenever we have a short exact sequence 
\[
0 \to A \to \Omega_X \to Q \to 0
\]
in $\Bc_K$, we will always write $\beta$ to denote the map from $A$ to $\Omega_X$ in the corresponding long exact sequence in $\Coh (X)$
\[
0 \to Q^{-1} \to A \overset{\beta}{\to} \Omega_X \to Q^0 \to 0
\]
where $Q^i := H^i(Q)$. 


\begin{lem}\label{lem:AG53-53-1}
Let $X$ be a smooth projective surface and $\omega$ an ample divisor on $X$. Suppose we have a short exact sequence in $\Bc_\omega$ 
\[
0 \to A \to E \to Q \to 0
\]
where $A, Q \neq 0$  and $\rho_\omega (A) = \rho_\omega (\Omega_X) = \rho_\omega (Q)$. 
Suppose also that $\omega c_1(E)>0, \Re Z_\omega (E)>0$, and $\ch_0(Q^0)>0$. Then we have the inequality
\[
- \frac{\Re Z_\omega (E)}{\Im Z_\omega (E)} > \frac{\ch_2(Q^0) - \tfrac{ (\omega c_1(E))^2}{2\omega^2}\ch_0(Q^{-1}) - \tfrac{\omega^2}{2}\ch_0(Q)}{\omega c_1(Q^0)}.
\]
\end{lem}

\begin{proof}
We loosely follow the argument in \cite[Theorem 6.1]{CLSY2}.  We begin by estimating an upper bound for $\ch_2(Q^{-1})$: if $Q^{-1}$ is nonzero, let $J_1, \cdots J_m$ denote the $\mu_\omega$-HN factors of $Q^{-1}$.  Then for each $i$, the Bogomolov-Gieseker inequality and the Hodge index theorem together give
\begin{equation}\label{eq:AG53-73-0}
  \ch_2 (J_i) \leq \frac{c_1(J_i)^2}{2\ch_0(J_i)} \leq \frac{1}{2\omega^2}\frac{ (\omega c_1(J_i))^2}{\ch_0(J_i)}.
\end{equation}
We also have
\begin{equation}\label{eq:AG53-73-1}
0 \leq \omega c_1(J_i[1]) \leq \omega c_1(Q^{-1}[1]) < \omega c_1(Q) \leq \omega c_1(E)
\end{equation}
where  the third inequality is strict since $Q^0$ is a coherent sheaf of nonzero rank in the torsion class $\Tc_\omega$, meaning $\omega c_1(Q^0)>0$.

From \eqref{eq:AG53-73-0} and \eqref{eq:AG53-73-1}, we now have
\[
\ch_2(Q^{-1}) = \sum_{i=1}^m \ch_2(J_i) < \frac{1}{2\omega^2} \ch_0(Q^{-1})(\omega c_1(E))^2
\]
and so 
\begin{align*}
-\Re Z_\omega (Q)  &= \ch_2(Q) - \tfrac{\omega^2}{2}\ch_0(Q) \\
&= \ch_2(Q^0) - \ch_2(Q^{-1}) - \tfrac{\omega^2}{2}\ch_0(Q) \\
&>  \ch_2(Q^0)- \frac{ (\omega c_1(E))^2}{2 \omega^2} \ch_0(Q^{-1}) - \tfrac{\omega^2}{2}\ch_0(Q).
\end{align*}

Since we are assuming $-\Re Z_\omega (E)<0$ and $\rho_\omega (E) = \rho_\omega (Q)$, it follows that $-\Re Z_\omega (Q)<0$ as well, and so (using $0 < \omega c_1(Q^0) \leq \omega c_1(Q)$)
\[
\frac{-\Re Z_\omega (E)}{\Im Z_\omega (E)} = \frac{-\Re Z_\omega (Q)}{\Im Z_\omega (Q)} = \frac{-\Re Z_\omega (Q)}{\omega c_1(Q)} > \frac{ \ch_2(Q^0)- \frac{ (\omega c_1(E))^2}{2 \omega^2} \ch_0(Q^{-1}) - \tfrac{\omega^2}{2}\ch_0(Q)}{\omega c_1(Q^0)}
\]
as claimed.
\end{proof}

\begin{prop}\label{prop:AG53-73-1}
Let $X$ be a smooth projective surface with ample canonical class $K$ and  $\tau_X \geq 0$.   Suppose $t \geq 1$ and we have  
 a short exact sequence in $\Bc_K$
\[
0 \to A \to \Omega_X \to Q \to 0
\]
where $A, Q \neq 0$,  $Q$ is $\sigma_\omega$-semistable, and $\rho_\omega (A)= \rho_\omega (\Omega_X) = \rho_\omega (Q)$.   Then $\ch_0(\image \beta)=2=\ch_0(A)$ and  $Q^{-1}=0$.
\end{prop}

Note that when the short exact sequence $0 \to A \to \Omega_X \to Q \to 0$ in $\Bc_K$ is a destabilizing sequence with respect to $\sigma_\omega$ in the sense of \ref{para:outermostwall-existence}, Proposition \ref{prop:AG53-73-1} already follows from Corollary  \ref{cor:AG53-70-1}(ii) and Lemma \ref{lem:AG54-24-1-der1}.  Nonetheless, we give a different argument for proving   Proposition \ref{prop:AG53-73-1} below, as some of the techniques will be useful for surfaces with $\tau_X < 0$.

\begin{proof}
Consider the long exact sequence of sheaves
\[
0 \to Q^{-1} \to A \overset{\beta}{\to} \Omega_X \to Q^0 \to 0.
\]
Suppose  $\ch_0(\image \beta)=1$.  Then   $\image \beta \cong M \otimes I_Z$ for some line bundle $M$ and    $I_Z$  the ideal sheaf of some 0-dimensional subscheme $Z$ of $X$ (say of length $n$).  It is easy to compute
\begin{align*}
      \ch (\Omega_X) &= (2, K, \tfrac{1}{2}\tau_X) \\
      \ch (\image \beta) &= \left(1, c_1(M), \tfrac{c_1(M)^2}{2} - n\right), \text{ and hence} \\
      \ch (Q^0) &= \left(1, K-c_1(M),  \tfrac{1}{2}\tau_X -  \tfrac{c_1(M)^2}{2} + n\right).
\end{align*}
Since $X$ is a minimal surface of general type, both $K^2, c_2(X)$ are positive integers, and so  
\[
\Re Z_\omega (\Omega_X) = -\tfrac{\tau_X}{2} + \omega^2 = (t^2-\tfrac{1}{2})K^2 +c_2(X)  >0
\]
since we are assuming $t \geq 1$.  Now Lemma \ref{lem:AG53-53-1} applies and we have  
\begin{align}
- \frac{\Re Z_\omega (\Omega_X)}{\Im Z_\omega (\Omega_X)} &> \frac{\ch_2(Q^0) - \tfrac{ (\omega c_1(\Omega_X))^2}{2\omega^2}\ch_0(Q^{-1}) - \tfrac{\omega^2}{2}\ch_0(Q)}{\omega c_1(Q^0)} \notag\\
&= \frac{\ch_2(Q^0) -  \frac{K^2}{2}\ch_0(Q^{-1}) - t^2\frac{K^2}{2}\ch_0(Q)}{tK (K-c_1(M))} \label{eq:AG53-85-1}
\end{align}

Under our assumptions, the object  $A$ is also $\sigma_\omega$-semistable with $\rho_\omega (A) =\rho_\omega (\Omega_X)$.  So by Corollary \ref{cor:AG53-70-1}, $\ch_0(A)$ is either 1 or 2.  Observe that
\begin{itemize}
    \item[(i)] If $\ch_0(A)=2$, then $\ch_0(Q^{-1})=1$ and $\ch_0(Q)=0$.
    \item[(ii)] If $\ch_0(A)=1$, then $Q^{-1}=0$, $A=\image \beta$, and $Q=Q^0$.
\end{itemize}
Also,  the numerator of  the  right-hand side of \eqref{eq:AG53-85-1} equals
\[
\begin{cases}
    \ch_2(Q^0) - \frac{K^2}{2} &\text{ in case (i)} \\
    \ch_2(Q^0) - t^2\frac{K^2}{2} &\text{ in case (ii)}
\end{cases}.
\]
Since $t \geq 1$, overall we have 
\[
- \frac{\Re Z_\omega (\Omega_X)}{\Im Z_\omega (\Omega_X)} > \frac{\ch_2(Q^0) - t^2\frac{K^2}{2}}{tK(K-c_1(M))}, 
\]
i.e.\  
\begin{align*}
    \frac{\tfrac{\tau_X}{2}-t^2K^2}{tK^2} &> \frac{\tfrac{\tau_X}{2} - \frac{c_1(M)^2}{2} + n - t^2\frac{K^2}{2}}{tK(K-c_1(M))} \\
    \Leftrightarrow  \left( \tfrac{\tau_X}{2}-t^2K^2 \right)  \frac{tK(K-c_1(M))}{tK^2} &> \tfrac{\tau_X}{2} - \tfrac{c_1(M)^2}{2} + n - t^2\tfrac{K^2}{2}    \\
    \Leftrightarrow  \left( \tfrac{\tau_X}{2}-t^2K^2 \right) \left(  1 - \tfrac{Kc_1(M)}{K^2} \right)  &>- \tfrac{c_1(M)^2}{2} + n  + \left( \tfrac{\tau_X}{2} - t^2 \tfrac{K^2}{2}  \right).
\end{align*}

Note that $\image \beta$ is a rank-one torsion-free sheaf in $\Tc_\omega$ and so $\omega c_1(\image \beta)=\omega c_1(M)>0$.  So by Lemma \ref{lem:AG53-52-1} we have $c_1(M)^2\leq 0$.  Therefore, to rule out  destabilizing sequences of the form described, it suffices to show 
\[
  \left( \tfrac{\tau_X}{2}-t^2K^2 \right) \left(  1 - \tfrac{Kc_1(M)}{K^2} \right)  \leq  \left( \tfrac{\tau_X}{2} - t^2 \tfrac{K^2}{2}  \right),
\]
i.e.
\[
  \left(  t^2K^2-\tfrac{\tau_X}{2} \right) \left(  1 - \tfrac{Kc_1(M)}{K^2} \right)  \geq  \left(  t^2 \tfrac{K^2}{2}-\tfrac{\tau_X}{2}  \right).
\]
By the $\mu_K$-semistability of $\Omega_X$, we have 
\[
Kc_1(M) = \mu_K (\image \beta) \leq \mu_K (\Omega_X) = \tfrac{K^2}{2}.
\]
Hence $ 1-\tfrac{Kc_1(M)}{K^2} \in [\tfrac{1}{2},1)$ and it suffices to show 
\begin{align*}
    \tfrac{1}{2}\left(  t^2K^2-\tfrac{\tau_X}{2} \right)  &\geq  \left(  t^2 \tfrac{K^2}{2}-\tfrac{\tau_X}{2}  \right) \\
  \Leftrightarrow \tau_X &\geq 0
\end{align*}
which we are assuming is true.  Hence $\ch_0(\image \beta)$ cannot be 1 and must be 2.  Since $\ch_0(A) \leq 2$, it follows that $\ch_0(A)=2$ and $Q^{-1}=0$.
\end{proof}

In the next proposition, we assume that we have some knowledge of the structure of the N\'{e}ron-Severi group of $X$, and use this information to write down constraints on the semistability of the cotangent bundle.

\begin{prop}\label{prop:AG53-114-1}
Let $X$ be a smooth projective surface with ample canonical class $K$.  Suppose $t \geq 1$ and  $0 \to A \to \Omega_X \to Q \to 0$ is a $\sigma_\omega$-destabilizing sequence in $\Bc_K$ corresponding to the outermost mini-wall in the sense of \ref{para:outermostwall-existence}.  Let  $r = \ch_0(A)$. Then 
\begin{itemize}
    \item[(i)]  Suppose $\ch_0 (\image \beta)=1$, and $\tfrac{K c_1(\image \beta)}{K^2} \leq b$ for some constant $b$.  Then whenever
    \[
      \left( \tfrac{\tau_X}{2} - t^2K^2\right) (1-b) \leq  \tfrac{\tau_X}{2} - \tfrac{K^2}{2} \left( (2-r)t^2 + (r-1)\right),
    \]
    such a destabilizing sequence does not exist.  If we take $b=\tfrac{1}{2}$, then this inequality simplifies to 
    \begin{equation}\label{eq:t0whenimbrk-1}
     -\tfrac{\tau_X}{4} \leq \tfrac{K^2}{2} (r-1)(t^2-1)
    \end{equation}
    In particular, if we take  $\tau_X  \geq 0$, such a destabilizing sequence does not exist.
    \item[(ii)] Suppose $\ch_0 (\image \beta)=2$.  Let $c = \min{\{ Kc_1(E) >0 : E \in \Coh (X)\}}$, and suppose $K^2=mc$ for some positive integer $m \geq 2$.    Then whenever
    \begin{equation}\label{eq:t0whenimbrk-2}
    t^2 \left( \tfrac{m}{m-1}\cdot \tfrac{r}{2}-1\right) >  \tfrac{m}{m-1}\left( \tfrac{r}{2}-\tfrac{3}{4}\right) - \tfrac{\tau_X}{2K^2}
    \end{equation}
    such a destabilizing sequence does not exist.
\end{itemize}
\end{prop}

\begin{proof}
(i) Writing $\image \beta =M \otimes I_Z$ as in the proof of Proposition \ref{prop:AG53-73-1}, we again have the inequality  from \eqref{eq:AG53-85-1}, i.e.\
\[
- \frac{\Re Z_\omega (\Omega_X)}{\Im Z_\omega (\Omega_X)} > \frac{\ch_2(Q^0) -  \frac{K^2}{2}\ch_0(Q^{-1}) - t^2\frac{K^2}{2}\ch_0(Q)}{tK (K-c_1(M))}.
\]
Note that we also have
\begin{align*}
  \ch_0(Q^{-1}) + t^2 \ch_0(Q) &= \ch_0(Q^{-1}) + t^2(\ch_0(Q^0)-\ch_0(Q^{-1}))  \\
  &=\ch_0(Q^{-1})(1-t^2) + t^2 \ch_0(Q^0) \\
  &= (r-1)(1-t^2)+t^2
\end{align*}
and so
\[
- \frac{\Re Z_\omega (\Omega_X)}{\Im Z_\omega (\Omega_X)} > \frac{\ch_2(Q^0) - \tfrac{K^2}{2}\left((r-1)(1-t^2)+t^2\right)  }{tK (K-c_1(M))}.
\]
With $n$ denoting the length of the 0-dimensional scheme $Z$  and following the argument  in the proof of Proposition \ref{prop:AG53-73-1}, the above inequality can be rewritten as 
\[
  \frac{\tfrac{\tau_X}{2}-t^2K^2}{tK^2} > \frac{\tfrac{\tau_X}{2} - \frac{c_1(M)^2}{2} + n - \tfrac{K^2}{2}\left( (r-1)(1-t^2)+t^2\right) }{tK^2 \left( 1-\tfrac{Kc_1(M)}{K^2} \right) }.
\]
Sine $\tfrac{\tau_X}{2}-t^2K^2 = (\tfrac{1}{2}-t^2)K^2-c_2(X)$ is negative for all $t \geq 1$, it follows that 
\[
\left( \tfrac{\tau_X}{2}-t^2K^2 \right) (1-b) > \tfrac{\tau_X}{2} + \left( -\tfrac{c_1(M)^2}{2} + n\right) - \tfrac{K^2}{2}\left( (r-1)(1-t^2)+t^2\right).
\]
Since $c_1(M)^2 \leq 0$ by Lemma \ref{lem:AG53-52-1}, in order to rule out such destabilizing sequences, we just need
\[
\left( \tfrac{\tau_X}{2}-t^2K^2 \right) (1-b) \leq \tfrac{\tau_X}{2}  - \tfrac{K^2}{2}\left( (2-r)t^2+(r-1)  \right)
\]
as claimed.  

Since $\Omega_X$ is $\mu_K$-semistable, we have $\mu_K(\image \beta) \leq \mu_K(\Omega_K)$, i.e.\ $Kc_1(M) \leq \tfrac{K^2}{2}$.  That is, we can take $b=\tfrac{1}{2}$, in which case the above inequality simplifies to 
\begin{align*}
-\tfrac{\tau_X}{4} &\leq \tfrac{K^2}{2} \left( (r-1)t^2 + (1-r)\right) \\
\Leftrightarrow -\tfrac{\tau_X}{4} &\leq \tfrac{K^2}{2} (r-1)(t^2-1).
\end{align*}


(ii) By the same argument as in the proof of Lemma \ref{lem:AG53-53-1}, we have an upper bound for $\ch_2(Q^{-1})$:
\[
\ch_2(Q^{-1}) \leq \frac{(\omega c_1(\Omega_X))^2}{2\omega^2}\ch_0(Q^{-1})=\tfrac{K^2}{2}(r-2).
\]
Applying Lemma \ref{lem:AG53-116-1} to $\image \beta$ gives $\ch_2 (\image \beta) \leq \tfrac{K^2}{4}$.  Hence 
\[
\ch_2(A) = \ch_2(Q^{-1})+\ch_2(\image \beta) \leq \tfrac{K^2}{4} (2r-3).
\]
By the assumption $\rho_\omega (\Omega_X)=\rho_\omega (A)$, we now have
\begin{align}
\frac{\tfrac{\tau_X}{2} - t^2K^2}{tK^2} &= \frac{\ch_2(A) - t^2K^2 \tfrac{r}{2}}{tKc_1(A)} \notag\\
&\leq \frac{\tfrac{K^2}{4}(2r-3) - t^2K^2 \tfrac{r}{2}}{tKc_1(A)} \notag\\
& = \frac{ K^2 \left( \tfrac{r}{2}(1-t^2) - \tfrac{3}{4}\right)  }{tKc_1(A)} \label{eq:AG54-40-1}
\end{align}

Note that we cannot have $Kc_1(Q)=0$ for $1>\rho_\omega (\Omega_X)=\rho_\omega (Q)$ and so  $0<Kc_1(A) =K(c_1(\Omega_X)-c_1(Q)) \leq (m-1)c$.  Also, the numerator on the right-hand side of \eqref{eq:AG54-40-1} is non-positive for $t \geq 1$.  As a result, we have 
\[
\tfrac{\tau_X}{2K^2} - t^2 \leq  \frac{ K^2 \left( \tfrac{r}{2}(1-t^2) - \tfrac{3}{4}\right)  }{(m-1)c} 
\]
i.e.\
\[
\tfrac{\tau_X}{2K^2} - t^2 \leq   \tfrac{m}{m-1}  \left( \tfrac{r}{2}(1-t^2) - \tfrac{3}{4}\right).
\]
To rule out destabilizing sequences of this type, we need
\begin{align*}
\tfrac{\tau_X}{2K^2} - t^2 &>   \tfrac{m}{m-1}  \left( \tfrac{r}{2}(1-t^2) - \tfrac{3}{4}\right) \\
\Leftrightarrow       t^2 \left( \tfrac{m}{m-1}\cdot \tfrac{r}{2}-1\right) &>  \tfrac{m}{m-1}\left( \tfrac{r}{2}-\tfrac{3}{4}\right) - \tfrac{\tau_X}{2K^2}.
\end{align*}
\end{proof}


\begin{prop}\label{prop:AG53-124-1}
Let $X$ be a smooth projective surface with ample canonical class $K$.  Suppose  $t \geq 1$, and  $0 \to A \to \Omega_X \to Q \to 0$ is a $\sigma_\omega$-destabilizing sequence in $\Bc_K$ corresponding to the outermost mini-wall  in the sense of \ref{para:outermostwall-existence}. Suppose   $\ch_0(\image \beta) =1$ and  $\ch_0(A) \leq 2$.  Then such a destabilizing sequence does not exist whenever 
\[
t^2 > \max{\{-\tfrac{\tau_X}{2}, -C_X+1 \}}.
\]
In particular, when $\tau_X \geq 0$ and $t \geq 1$, such a destabilizing sequence does not exist. 
\end{prop}

\begin{proof}
We divide into two cases according to whether $\ch_0(A)=1$ or $2$.  

\noindent
\textbf{Case I}: $\ch_0(A)=1$.  In this case, $Q^{-1}$ vanishes and $A=\image \beta$ is a subsheaf of $\Omega_X$.  By Proposition \ref{prop:AG53-73-1}, we cannot have $\tau_X \geq 0$, and so let us assume $\tau_X < 0$ from now on.  As in the proof of Proposition \ref{prop:AG53-73-1}, we can write $\image \beta = M \otimes I_Z$ for some line bundle $M$ and 0-dimensional subscheme $Z \subset X$ of length $n$.  Then 
\[
\ch(Q)=  \left(1, K-c_1(M),  \tfrac{1}{2}\tau_X -  \tfrac{c_1(M)^2}{2} + n\right)
\]
and  the assumption $\rho_\omega (\Omega_X)=\rho_\omega (Q)$ subsequently gives
\begin{align*}
  \frac{\tfrac{\tau_X}{2}-t^2K^2}{tK^2} &= \frac{ \left( \tfrac{\tau_X}{2}-\tfrac{c_1(M)^2}{2}+n\right) - \tfrac{t^2K^2}{2}}{tK(K-c_1(M))} \\
  \Leftrightarrow   \left( \tfrac{\tau_X}{2} - t^2K^2\right) \left( 1 - \tfrac{Kc_1(M)}{K^2}\right) &= \left( \tfrac{\tau_X}{2} - \tfrac{c_1(M)^2}{2} +n\right) - \tfrac{t^2K^2}{2} \\
  \Leftrightarrow  t^2K^2 \left( \tfrac{1}{2} - \tfrac{Kc_1(M)} {K^2}\right) &= -\tfrac{\tau_X}{2}\cdot \tfrac{Kc_1(M)}{K^2}  + \tfrac{c_1(M)^2}{2} - n.
\end{align*}

Since  $\Omega_X$ is  $\mu_K$-semistable, we have  $\mu_K(A) \leq \mu_K(\Omega_X)$, i.e.\
\begin{equation}\label{eq:AG53-124-0.5}
0 <  \tfrac{Kc_1(M)}{K^2} \leq \tfrac{1}{2}.
\end{equation}
Hence 
\begin{equation}\label{eq:AG53-124-1}
t^2 = \frac{-\tfrac{\tau_X}{2}\cdot \tfrac{Kc_1(M)}{K^2}  + \tfrac{c_1(M)^2}{2} - n}{K^2 \left( \tfrac{1}{2}-\tfrac{Kc_1(M)}{K^2} \right)} \leq \frac{-\tfrac{\tau_X}{2}\cdot \tfrac{Kc_1(M)}{K^2}}{K^2 \left( \tfrac{1}{2}-\tfrac{Kc_1(M)}{K^2} \right)}
\end{equation}
since $c_1(M)^2 \leq 0$ and $n \geq 0$.

Now consider the function 
\[
h(x) := -\frac{\tau_X}{2K^2}\cdot \frac{x}{\tfrac{1}{2}-x}.
\]
Note that $h(\tfrac{Kc_1(M)}{K^2})$ is the right-hand side of \eqref{eq:AG53-124-1}.  Owing to  \eqref{eq:AG53-124-0.5}, we are only interested in  $h(x)$ as a real-valued function over the interval $(0,\tfrac{1}{2})$.  Note that $\frac{d}{dx} \frac{x}{(1/2)-x}>0$ for all $x \in (0,\tfrac{1}{2})$, i.e.\ $h(x)$ is an increasing function on $(0,\tfrac{1}{2})$.  In addition, with $c$ and $m$  as defined in Proposition \ref{prop:AG53-114-1}, then  $Kc_1(M)$ is a positive integer multiple of $c$ while also $Kc_1(M) < \tfrac{K^2}{2}=\frac{m}{2}c$.  Therefore, regardless of the parity of the positive integer $m$, we have $Kc_1(M) \leq (\tfrac{m}{2}-\tfrac{1}{2})c$ and hence
\[
\frac{Kc_1(M)}{K^2} \leq \frac{\tfrac{m}{2}-\tfrac{1}{2}}{m} = \tfrac{1}{2}-\tfrac{1}{2m}.
\]

As a result, the described destabilizing sequence cannot exist when 
\[
t^2 > h\left( \tfrac{1}{2}-\tfrac{1}{2m} \right) = -\frac{\tau_X}{2K^2} (m-1).
\]
Since  $m\leq K^2$, it follows that such a destabilizing sequence cannot exist when 
\[
t^2 > -\frac{\tau_X}{2}.
\]

\noindent
\textbf{Case II}: $\ch_0(A)=2$.  Using \eqref{eq:t0whenimbrk-1} with $r=2$, we know that such a destabilizing sequence cannot exist when 
\[
 -\frac{\tau_X}{4} \leq \frac{K^2}{2}(t^2-1),
\]
i.e.\ when 
\[
t^2 \geq  - \frac{\tau_X}{2K^2}+1 =-C_X+1.
\]
\end{proof}

\begin{prop}\label{prop:AG53-122-1}
Let $X$ be a smooth projective surface with ample canonical class $K$ and $\tau_X \leq 0$.  Suppose  $t \geq 1$ and  $0 \to A \to \Omega_X \to Q \to 0$ is a $\sigma_\omega$-destabilizing sequence in $\Bc_K$ corresponding to the outermost mini-wall in the sense of \ref{para:outermostwall-existence}.  Suppose $Q^{-1}=0$ and    $\ch_0(A) \leq 2$.  Then such a destabilizing sequence does not exist whenever 
\[
t^2 \geq c_2(X) - \tfrac{3}{8}.
\]
\end{prop}


\begin{proof}
Under our assumptions,  $A$ is a subsheaf of $\Omega_X$ and $Q=Q^0$ is a torsion sheaf.  The assumption $\rho_\omega (\Omega_X)=\rho_\omega (Q)$ then means $Q$ must be supported in dimension 1, and so $0< Kc_1(A)< Kc_1(\Omega_X) = K^2$.  We also have     
\[
\ch_2(A) \leq  \tfrac{(Kc_1(A))^2}{4K^2}
\]
from  Lemma \ref{lem:AG53-116-1}. 

Now, the assumption $\rho_\omega (\Omega_X)=\rho_\omega (A)$ gives
\[
\frac{\tfrac{\tau_X}{2}-t^2K^2}{tK^2} = \frac{\ch_2(A) - t^2K^2}{tKc_1(A)} \leq \frac{\tfrac{(Kc_1(A))^2}{4K^2}-t^2K^2}{tKc_1(A)}
\]
which can be rewritten as
\begin{align*}
    \tfrac{\tau_X}{2} - t^2K^2 &\leq  \tfrac{Kc_1(A)}{4} - t^2\tfrac{(K^2)^2}{Kc_1(A)}  \\
    \Leftrightarrow  \left( \tfrac{(K^2)^2}{Kc_1(A)} - K^2 \right) t^2 &\leq \tfrac{Kc_1(A)}{4}-\tfrac{\tau_X}{2} 
\end{align*}
i.e.\   
\begin{equation}\label{eq:AG53-122-1}
    t^2 \leq \frac{1}{4K^2}\cdot  \frac{Kc_1(A)-2\tau_X}{\left( \tfrac{K^2}{Kc_1(A)}-1 \right) }.
\end{equation}
Let us now set
\[
b (x) :=  \frac{x-2\tau_X}{\left(\tfrac{K^2}{x}-1\right)},
\]
so that the right-hand side of \eqref{eq:AG53-122-1} equals $\tfrac{1}{4K^2}b(Kc_1(A))$.  Since $0< Kc_1(A) < K^2$ from earlier, we will think of $b(x)$ as a real-valued function on the interval $(0,K^2)$.

To rule out the destabilizing sequences that we are considering, we need $t^2 > \tfrac{1}{4K^2}b(Kc_1(A))$. So we will now try to find an upper bound for $b(x)$.

Observe that 
\[
\frac{d}{dx} b(x) = \frac{1}{\left( \tfrac{K^2}{x}-1\right)^2 x^2} n(x)
\]
where 
\[
n(x) := -x^2 + 2K^2 x - 2\tau_X K^2.
\]
The quadratic polynomial $n(x)$ in $x$ has discriminant
\[
\mathrm{disc}(n(x)) = 4(K^2)^2 -8\tau_XK^2 = 4K^2 (K^2-2\tau_X) = 4K^2(-K^2+4c_2(X)),
\]
which is positive since the BMY inequality gives $c_2(X) \geq \tfrac{1}{3}K^2$. Hence $n(x)$ has two distinct real roots, namely
\[
K^2 \pm \sqrt{K^2(-K^2+4c_2(X))}.
\]
Notice that the larger root of $n(x)$ is strictly larger than $K^2$.  On the other hand, the smaller root of $n(x)$
\[
K^2 - \sqrt{K^2(-K^2+4c_2(X))}
\]
is non-positive if and only if 
\begin{align*}
 K^2 &\leq \sqrt{K^2(-K^2+4c_2(X))} \\
 \Leftrightarrow  K^2 &\leq 2c_2(X) \\
 \Leftrightarrow  \tau_X &\leq 0.
\end{align*}

Therefore, under our assumption $\tau_X \leq 0$,  $\tfrac{d}{dx} b(x)$ is always nonnegative on the interval $(0,K^2)$, meaning $b(Kc_1(A))$ has an upper bound given by $b((m-1)c)$ where $c$ and $m$  as defined in Proposition \ref{prop:AG53-114-1}.  Thus 
\begin{align*}
\tfrac{1}{4K^2}b(Kc_1(A)) &\leq \tfrac{1}{4K^2}b(x)  \text{\quad with $x=(m-1)c$} \\
&= \tfrac{1}{4mc} \frac{ (m-1)c - 2\tau_X}{\tfrac{m}{m-1}-1} \\
&= \tfrac{1}{4}\cdot \tfrac{m-1}{m}\cdot \tfrac{1}{c} \left( -(m+1)c + 4c_2(X) \right) \\
&< c_2(X) - \tfrac{1}{4}\cdot \tfrac{m^2-1}{m} \text{\quad since $c \in \mathbb{Z}^+$} \\
&\leq c_2(X) - \tfrac{3}{8}
\end{align*}
where the last inequality holds because when $m=1$ we already know $\Omega_X$ is $\sigma_\omega$-stable for all $t>0$, and when $m \geq 2$ we have $\tfrac{m^2-1}{m}=m-\tfrac{1}{m} \geq \tfrac{3}{2}$.  As a result,  whenever $t^2 \geq c_2(X)-\tfrac{3}{8}$, a destabilizing sequences as described cannot exist.
\end{proof}

We end this section with a our main result on the semistability of the cotangent bundle on an arbitrary surface with ample canonical class, where the estimate of the outermost mini-wall is given solely in terms of the invariants $K_X^2$ and $c_2(X)$ of the surface $X$.  


\begin{thm}\label{thm:cotangentbld-main-1}
Let $X$ be a smooth projective surface with ample canonical class $K$ such that the cotangent bundle $\Omega_X$ is twisted $K$-Gieseker semistable.  Then the cotangent bundle $\Omega_X$ is $\sigma_\omega$-semistable under the following conditions 
\begin{itemize}
    \item[(i)] When $\tau_X \geq 0$:   
    \[
 t^2 >  \tfrac{K^2-1}{4K^2} \left( 4c_2(X) - K^2-1\right); 
\]
in particular, when $t^2 >   \tfrac{K^2}{4}$.
    \item[(ii)] When $\tau_X \leq 0$:  
    \[
    t^2 >  c_2(X)-\tfrac{3}{8}.
    \]
\end{itemize}
\end{thm}


\begin{proof}
Since $\Omega_X$ is twisted $K$-Gieseker semistable, it is  is $\sigma_\omega$-semistable for $t \gg 0$ \cite[Lemma 4.8]{LM2}.  Therefore, unless $\Omega_X$ is $\sigma_\omega$-semistable for all $t>0$, it makes sense to consider a $\sigma_\omega$-destabilizing sequence  $0 \to A \to \Omega_X \to Q \to 0$ in $\Bc_K$ in the sense of  \ref{para:outermostwall-existence}.

(i) This is Theorem \ref{thm:taunn-mainboundv2}. 


(ii) Suppose $t \geq 1$.  When $\tau_X\leq 0$, we have $C_X\leq $ and so $\max{\{ -\tfrac{C_X}{2}, \tfrac{1-6C_X}{3}\}}=\tfrac{1-6C_X}{3}=-2C_X+\tfrac{1}{3}$.  By Proposition \ref{prop:bounds-t-cases}, whenever $t^2 >  -2C_X+\tfrac{1}{3}$ we have $\ch_0(A) \leq 2$ and so $\ch_0 (\image \beta) \leq 2$.  By Proposition \ref{prop:AG53-122-1},   such a destabilizing sequence with $\ch_0(\image \beta)=2$ cannot exist when $t^2 \geq  c_2(X) - \tfrac{3}{8}$; on the other hand,  by Proposition \ref{prop:AG53-124-1}, such a sequence with $\ch_0 (\image \beta)=1$ cannot exist when $t^2 > \max{\{ -\tfrac{\tau_X}{2}, -C_X+1\}}$. 
Hence $\Omega_X$ is $\sigma_\sigma$-semistable when 
  \begin{equation}\label{eq:nonposindexgenbounds}
    t^2 > \max{\left\{ 1,  c_2(X)-\tfrac{3}{8}, - \tfrac{\tau_X}{2},  -C_X+1, -2C_X+\tfrac{1}{3} \right\}}.
    \end{equation}

We will now show that the cotangent bundle $\Omega_X$ is in fact $\sigma_\omega$-semistable whenever $t^2 > c_2(X)-\tfrac{3}{8}$.
Let us begin with the following observations:
\begin{itemize}
    \item $-\tfrac{\tau_X}{2} = c_2(X)-\tfrac{1}{2}K^2 \leq c_2(X) - \tfrac{3}{8}$.
    \item As noted in the proof of Theorem \ref{thm:taunn-mainboundv2}, $c_2(X)$ is at least 3, and so $c_2(X)-\tfrac{3}{8}>1$.
\end{itemize}

For the calculations below, let us write $x=c_2(X)$ and $y=K^2$.  Note that
\begin{align*}
    c_2(X) - \tfrac{3}{8} &> -2C_X + \tfrac{1}{3} \\
    \Leftrightarrow  x - \tfrac{3}{8} &> -2 \left( \tfrac{1}{2} - \tfrac{x}{y} \right) + \tfrac{1}{3} \\
    \Leftrightarrow  1+ \tfrac{7}{24x} &> \tfrac{2}{y},
\end{align*}
which is true whenever $y=K^2\geq 2$.  When $y=K^2=1$, we already know from Lemma \ref{lem:surf-m-is-1} that $\Omega_X$ is $\sigma_\omega$-semistable for all $t>0$, and in particular is $\sigma_\omega$-semistable whenever $t> c_2(X)-\tfrac{3}{8}$ anyway. 

On the other hand, 
\begin{align*}
     c_2(X) - \tfrac{3}{8} &> -C_X+1 \\
     \Leftrightarrow 1-\tfrac{1}{y} &> \tfrac{7}{8x}.
\end{align*}
From the last two paragraphs, we know  $x \geq 3$ and that we can assume $y \geq 2$, which in turn imply 
\[
1-\tfrac{1}{y} \geq \tfrac{1}{2} > \tfrac{7}{24} \geq \tfrac{7}{8x}
\]
and so $c_2(X) - \tfrac{1}{8} > -C_X+1$ holds.

Overall, we see that $\Omega_X$ is $\sigma_\omega$-semistable whenever $t^2 > c_2(X) -\tfrac{3}{8}$.
\end{proof}

\section{Example:  surfaces of general type with minimal $c_2$}\label{sec:pg0K29}


In this and the next section, we specialize to specific types of surfaces with an ample canonical class.  We will see that  a more detailed understanding of the geometry of the surface allows us to obtain  sharper bounds for the stability of $\Omega_X$ than in  previous sections.

Recall that for a smooth projective surface $X$, its topological Euler characteristic $e(X)$ is the same as $c_2(X)=c_2(T_X)$ \cite[1.1]{yeung2013classification}.  Given a minimal surface of general type $X$, the BMY inequality together with \eqref{eq:mod12} imply that the minimum possible value of $c_2(X)$ is 3.  In Yeung's work \cite[Theorem 1]{yeung2013classification}, it was shown that there are only two types of minimal surfaces of general type with $c_2(X)=3$: the fake projective planes  and the Cartwright-Steger surface; there are 100 fake projective planes and exactly 1 Cartwright-Steger surface up to isomorphism.  In this section, we study the  stability of the cotangent bundle with respect to $\sigma_{tK}$, $t>0$, for both the Cartwright-Steger surface and the fake projective planes.

\paragraph \label{para:CSsurf} A \emph{Cartwright-Steger (CS) surface} is a minimal surface of general type $X$ satisfying $c_2(X)=3, K_X^2=9$, with geometric genus $p_g(X)=1$ and irregularity $q=1$.  The bicanonical class $2K_X$ on a CS surface is very ample  \cite[Remark 6.1]{borisov2020explicit}, and so $K_X$ itself is ample.  

The work of Cartwright,  Koziarz, and Yeung \cite{cartwright2014cartwright} describes basic properties of $\mathrm{NS}(X)$ for a CS surface $X$.  In particular, they show that  $\NS(X)$ has a basis consisting of divisor classes $E_1, E_2, C$ where  
\[
E_1^2=E_2^2=5, \text{\quad} C^2=-1,  \text{\quad}E_1.E_2=13,  \text{\quad}E_1.C=11,  \text{\quad}E_2.C=7
\]
and
\[
K_X \equiv \tfrac{1}{2}(E_1+E_2),  \text{\quad} K_X^2 = 9
\]
(see \cite[Lemma 7 and Proposition 10]{cartwright2014cartwright}).  As a result, we have 
\[
K_X.E_1 = 9 = K_X.E_2, \text{\quad} K_X.C=9
\]
and hence $c=9, m=1$ where $c, m$ are as in \ref{para:defn-cm}.  Then Lemma \ref{lem:surf-m-is-1} applies, and we obtain

\begin{prop}\label{pro:CSsurf}
Let $X$ be the Cartwright-Steger surface.  Then the cotangent bundle $\Omega_X$ is $\sigma_{tK}$-semistable for all $t>0$.
\end{prop}


\paragraph A \emph{fake projective plane (FPP)} $X$ is a minimal surface of general type with $p_g=0, K_X^2=9$.  Since $\ch_0(\Omega_X)=2$ and $Kc_1(\Omega_X)=9$ are coprime, the cotangent bundle $\Omega_X$ is $\mu_K$-stable and non-split on a fake projective plane.

For every fake projective plane $X$, the Picard rank is 1, and we can find an ample generator $L$ for the  N\'{e}ron-Severi group such that $K_X \equiv 3L$ and $L^2=1$ \cite{yeung2010classification}.  Moreover, since  $p_g(X)=0$, from \ref{para:sgt-basics} we have $p_a=0$ and Hirzebruch-Riemann-Roch gives $12=K_X^2+c_2(X)$, and so $c_2(X)=3$.  Then $\tau_X = K_X^2 - 2c_2(X)=3$.  That is, the fake projective planes achieve equality in the BMY inequality $K_X^2 \leq 3c_2(X)$ for minimal surfaces of general type.  Since $X$ has Picard rank 1, there are no negative curves on $X$.


\begin{thm}\label{thm:AG54-2-1}
Let $X$ be a fake projective plane.  Let $L$ be an ample generator of $\mathrm{NS}(X)$ such that $L^2=1$ and $K_X \equiv 3L$.  Then $\Omega_X$ is $\sigma_{tK}$-semistable for all $t>0$.
\end{thm}

\begin{proof}
Let us write $K=K_X$.  Suppose we have a destabilizing sequence in $\Bc_K$
\[
0 \to A \to \Omega_X \to Q \to 0
\]
corresponding to the outer-most mini-wall in the sense of \ref{para:outermostwall-existence}.   Consider the long exact sequence of sheaves
\[
0 \to Q^{-1} \to A \overset{\beta}{\to} \Omega_X \to Q^0 \to 0
\]
where $Q^i := H^i(Q)$.  Since $\beta \neq 0$,  the rank of $\ch_0(\image \beta)$ must be either 1 or 2.

We claim that the rank of $A$ is at most 2.  To see this, recall from the proof of Corollary \ref{cor:AG53-70-1} that 
\[
\ch_0(A) < \sqrt{(Cs^2-1)^2+s^2}-(Cs^2-1) 
\]
where $C := C_X=\tfrac{\tau_X}{2K^2}$ and $s=\tfrac{1}{t}$.  Since  $C=\tfrac{1}{6}$ for a FPP, we just need 
\begin{align*}
\sqrt{(\tfrac{1}{6}s^2-1)^2+s^2}-(\tfrac{1}{3}s^2-1) & \leq 3 \\
\Leftrightarrow  (\tfrac{1}{6}s^2-1)^2+s^2 &\leq (3+(\tfrac{1}{3}s^2-1))^2 \\
\Leftrightarrow \tfrac{1}{12}s^4 + \tfrac{2}{3}s^2 + 3 &\geq 0
\end{align*}
which is true for any $t>0$.  Hence $\ch_0(A) \leq 2$.  By Lemma \ref{lem:AG54-29-1}, we have $\ch_0(\image \beta)=2$, and so $Q^{-1}=0$ and $A$ is a rank-2 subsheaf of $\Omega_X$.

From the previous paragraph, we know $Q$ must be a 1-dimensional torsion sheaf, and so $Kc_1(Q)>0$ which forces $0<Kc_1(A) \leq 6$.   
From \eqref{eq:AG54-26-1} in the proof of Proposition \ref{pro:AG54-24-1}, we have
\begin{align*}
   \rho_{t\omega}(A) - \rho_{t\omega}(\Omega_X) = \frac{1}{tKc_1(A)\cdot K^2}\left( 2t^2K^2 (\mu_K(A)-\mu_K(\Omega_X) ) + K^2\ch_2(A) - \tfrac{\tau_X}{2}Kc_1(A)    \right).
\end{align*}
Therefore, in order to achieve  $\rho_{t\omega} (A)-\rho_{t\omega}(\Omega_X) < 0$, we need
\begin{equation}\label{eq:AG54-30-1}
2t^2K^2 (\mu_K(A)-\mu_K(\Omega_X) ) + K^2\ch_2(A)-\tfrac{\tau_X}{2}Kc_1(A) < 0. 
\end{equation}

Since $Kc_1(A)=3$ or $6$, we divide into two cases.

\noindent
\textbf{Case 1: $Kc_1(A)=3$.}  In this case, $c_1(A)=L$ and so $c_1(A)^2=1$.  Then \eqref{eq:AG54-30-1} becomes
\[
-6t^2 + \ch_2(A) - \tfrac{1}{2} < 0.
\]
Since $L$ is ample, $c_1(A)$ is effective, and so by Lemma \ref{lem:Miyrmkcoro}, we have $\ch_2(A) \leq \tfrac{1}{6}$.  Hence  the above inequality holds for all $t>0$.

\noindent
\textbf{Case 2: $Kc_1(A)=6$.} In this case, $c_1(A)=2L$ and $c_1(A)^2=4$.  The inequality  \eqref{eq:AG54-30-1} becomes 
\[
-3t^2 + \ch_2(A) -1 < 0.
\]
As in Case 1,  Lemma \ref{lem:Miyrmkcoro} applies and we have $\ch_2(A) \leq \tfrac{2}{3}$, and so  the above inequality holds for all $t>0$.

The above argument shows that the outer-most mini-wall for $\Omega_X$ does not exist in the family $\{\sigma_{tK}\}_{t>0}$, and so $\Omega_X$ is $\sigma_{tK}$-semistable for all $t>0$.
\end{proof}

\begin{rem}
In the proof of Theorem \ref{thm:AG54-2-1}, the rank-2 subsheaf $A$ of $\Omega_X$ must be $\mu_K$-semistable.  Suppose $A$ is not $\mu_K$-semistable; then it must have a 2-term HN filtration 
\[
0 \neq F_0 \subsetneq F_1 = A
\]
where the HN factors $F_0, F_1/F_0$ are both rank-one torsion-free sheaves with $Kc_1 >0$.  Since $K\equiv 3L$ and $\mu_K(F_0) > \mu_K(F_1/F_0)$, we must have $Kc_1(F_0)=\mu_K(F_0)\geq 6$ while $Kc_1(F_1/F_0)=\mu_K (F_1/F_0)\geq 3$.  Since we also have  $0< Kc_1(A) \leq Kc_1(\Omega_X)=K^2=9$, it follows that $Kc_1(A)=9$ and $Kc_1(Q)=0$.  This forces $Q$ to be a torsion sheaf supported in dimension 0, and so $\rho_{t\omega}(Q)=1$ for all $t>0$, contradicting our assumption that $\rho_{t\omega}(\Omega_X)=\rho_{t\omega}(Q)$ at the outer-most mini-wall.  Hence $A$ must be $\mu_K$-semistable.
\end{rem}

\section{Example: minimal surfaces of general type with $p_g=0, K_X^2=8$}\label{sec:pg0K28}


A minimal surface of general type $X$ with $p_g=0$ and $K^2_X=8$ is called a \emph{fake quadric}.  For every fake quadric $X$, it is known that the canonical class $K_X$ is ample (e.g. see \cite[Section 2]{CCCK25}).  In this section, we study the stability of the cotangent bundle on a fake quadric with respect to $\sigma_{tK}$ for $t>0$ where $K:=K_X$.


As mentioned in \ref{para:outermostwall-existence}, the following lemma ensures that on a fake quadric, the cotangent bundle is at least $\sigma_{tK}$-semistable for $t \gg 0$.

\begin{lem}\label{lem:AG54-21-1}
Let $X$ be a minimal surface of general type with $p_g=0$ and $K^2_X=8$.  Then the cotangent bundle $\Omega_X$ is twisted Gieseker semistable with respect to $K_X$.
\end{lem}

\begin{proof}
If $\Omega_X$ is non-split, then it is $\mu_K$-stable and hence twisted Gieseker semistable with respect to $K_X$.

Suppose $\Omega_X = L \oplus L'$ for some line bundles $L, L'$.  Then by the $\mu_K$-semistability of $\Omega_X$, we have
\[
  Kc_1(L)=Kc_1(L')=\tfrac{1}{2}Kc_1(\Omega_X)=\tfrac{1}{2}K^2.
\]
Since $K^2>0$,  the Bogomolov inequality and Lemma \ref{lem:AG53-52-1} together give $\ch_2(L), \ch_2(L') \leq 0$.  Since $\tau_X=0$, we also have $\ch_2(\Omega_X)=\tfrac{1}{2}\tau_X=0$ and hence $\ch_2(L)=0=\ch_2(L')$.  That is, $L$ and $L'$ have the same $\ch_0, K\ch_1, \ch_2$, and hence the same Hilbert polynomial with respect to $K$.  Since every line bundle is twisted Gieseker semistable, and $L, L'$ have the same Hilbert polynomial, their direct sum $L \oplus L'$ is again twisted Gieseker semistable.  
\end{proof}

\paragraph[Background on fake quadrics] 
\label{para:fakequadric}
\label{para:summary-yang} 
In this section, we collect some basic results of fake quadrics from \cite{CCCK25}.
Fake quadrics  can be divided into two classes according to the parity of the N\'{e}ron-Severi lattice: 
\begin{itemize}
    \item Fake quadrics of \emph{even type}: In this case, the fake quadric $X$ has N\'{e}ron-Severi lattice generated by two divisors $Q_1, Q_2$ satisfying
    \[
    Q_1^2=0=Q_2^2, \text{\quad} Q_1 Q_2=1
    \]
    and $K_X \equiv 2Q_1+2Q_2$ \cite[Section 4]{CCCK25}.  There are no negative curves on $X$ by \cite[Proposition 4.2(1)]{CCCK25}. 
    \item Fake quadrics of \emph{odd type}: In this case, the N\'{e}ron-Severi lattice of the fake quadric $X$ is generated by divisors $Q_1, Q_2$ satisfying
    \[
Q_1^2=1, \text{\quad} Q_2^2=-1, \text{\quad} Q_1Q_2=0
    \]
    and $K_X \equiv 3Q_1-Q_2$ \cite[Section 6]{CCCK25}.  For a fake quadric of odd type, it is not known whether  negative curves exist on $X$ in general.
\end{itemize}

We denote the ample, nef, and effective cones of $X$ by $\mathrm{Amp}(X)$, $\mathrm{Nef}(X)$, and $\mathrm{Eff}(X)$, respectively, and write $\overline{\mathrm{Eff}}(X)$ for the closure of the effective cone.
For fake quadrics of even type, the structures of the ample cone and effective cone is summarized in the following proposition. 
\begin{prop}\cite[Corollary 4.4]{CCCK25}
\label{prop:ANNeven}
Let $X$ be a fake quadric of even type. Then 
\[
\begin{split}
\mathrm{Amp}(X)&= \{a[Q_1]+b[Q_2]|a, b\in \mathbb{R}_{>0}\},\\
\mathrm{Nef}(X)&= \{a[Q_1]+b[Q_2]|a, b\in \mathbb{R}_{\geq 0}\},\\
\mathrm{Nef}(X)&=\overline{\mathrm{Eff}}(X).
\end{split}
\]
\end{prop}

The following result is contained in the proof of \cite[Theorem 4.6]{CCCK25}:

\begin{prop}\cite{CCCK25}
Let $X$ be a fake quadric of odd type with no negative curves. 
Then 
\[
\begin{split}
\mathrm{Nef}(X)&= \{a[Q_1+Q_2]+b[Q_1-Q_2]|a, b\in \mathbb{R}_{\geq 0}\},\\
\mathrm{Nef}(X)&=\overline{\mathrm{Eff}}(X).
\end{split}
\]
\end{prop}

\begin{rem} 
\label{rem:congmod2}Let $A$ be a torsion free sheaf on $X$. By Hirzebruch–Riemann–Roch, we have 
\[
c_1(A)^2-Kc_1(A)=2(\chi(\det(A)-\chi(\cO_X))).
\]
As a result, we have so $c_1(A)^2 \equiv Kc_1(A)$ mod $2$. 
\end{rem}
\begin{thm}
\label{thm:wallfakequadric}
Let $X$ be a smooth projective surface with ample canonical class $K$,  $p_g=0$, and $K^2=8$.    Then  $\Omega_X$ is $\sigma_{tK}$-semistable for all $t^2 \geq \frac{1}{2}$. 
\end{thm}

\begin{proof}
Suppose we have a destabilizing sequence
\[
0 \to A \to \Omega_X \to Q \to 0
\]
that corresponds to the outer-most mini-wall in the sense of  \ref{para:outermostwall-existence}. Since $\tau_X=0$, Corollary 5.7 (ii) implies that $A$ in $\Bc_K$ of rank at most $2$ when $t^2\geq 1/3$. Furthermore, Lemma \ref{lem:AG54-24-1-der1} implies that $A$ is of rank at least $2$. From now on we assume $t^2\geq \frac{1}{3}$ so that  $A$ has rank exactly $2$.   

We claim that the object $A$ is actually a subsheaf of $\Omega_X$.  Consider the long exact sequence of sheaves
\[
0 \to Q^{-1}\to A \overset{\beta}{\to} \Omega_X \to Q^0 \to 0.
\]

Suppose for now that  $\ch_0( \image \beta) =1$.  Then by Lemma \ref{lem:AG53-52-1} we have $c_1(\image \beta)^2 \leq 0$ and hence $\ch_2(\image \beta) \leq 0$ by the BG inequality.  Hence $\ch_2(A) \leq \ch_2(Q^{-1})$.  Note, however, that $\ch_0(Q^{-1})=1$ and so $Q^{-1}$ is $\mu_K$-stable.  In addition, by the $\mu_K$-semistability of $\Omega_X$ we have $Kc_1(\image \beta) \leq 4$ and so $Kc_1(Q^{-1}) = Kc_1(A)-Kc_1(\image \beta) \geq -3$ (noting $Kc_1(A)>0$).  Applying the BG inequality to $Q^{-1}$ now gives
\[
\ch_2(A) \leq \ch_2(Q^{-1})\leq \frac{(K\ch_1(Q^{-1}))^2}{2K^2}\leq \frac{9}{16}.
\]
Since $\ch_2(A)$ must be in $\tfrac12 \mathbb{Z}$, we obtain $\ch_2(A) \leq \tfrac12$.  

On the other hand, since $\ch_0(Q)=0$, we must have $Kc_1(Q)>0$ or else $Q$ is supported in dimension $0$ and we have $\rho_\omega (A) < \rho_\omega (Q)$, contradicting the assumption that $A$ is a destabilizing subobject at a wall.  Hence $\mu_K (A) < \mu_K (\Omega_X)$.  Since $\rho_\omega (A)=\rho_\omega (\Omega_X)$, from \eqref{eq:AG54-26-1}, we know 
\[
t^2K^2\ch_0(A)(\mu_K(A)-\mu_K(\Omega_X))+K^2\ch_2(A) = 0
\]
or equivalently
\begin{equation}\label{eq:AG54-29-1b}
t^2 = \frac{\ch_2(A)}{2(\mu_K(\Omega_X)-\mu_K(A))}=\frac{\ch_2(A)}{K^2-Kc_1(A)} \leq \frac18
\end{equation}
where the last inequality follows from the upper bound of $\ch_2(A)$ in the previous paragraph and the fact that $Kc_1(A) \leq Kc_1(\image \beta) \leq 4$.  The inequality \eqref{eq:AG54-29-1b} now contradicts our assumption $t^2 \geq \tfrac13$, and so $\ch_0(\image \beta)$ must be $2$.  Thus $Q^{-1}=0$ and $A$ is a rank-2 subsheaf of $\Omega_X$.

We now try to bound $\ch_2(A)$ by dividing into different cases.  Note that we still have the wall equation 
\[
t^2 = \frac{\ch_2(A)}{2(\mu_K(\Omega_X)-\mu_K(A))}=\frac{\ch_2(A)}{K^2-Kc_1(A)} 
\]
from \eqref{eq:AG54-29-1b}.  

\noindent
\textbf{If $1 \leq Kc_1(A) \leq 6$:}  In this case, we can use Lemma \ref{lem:AG53-116-1} which gives
\[
\ch_2(A) \leq \frac{(Kc_1(A))^2}{4K^2}=\frac{(Kc_1(A))^2}{32}.
\]
For each value of $Kc_1(A)$ between $1$ and $6$ inclusive, we can now compute an upper bound for $\ch_2(A)$ while noting it must be a half integer, and then use the wall equation to obtain an upper bound for $t^2$:
\begin{itemize}
    \item $Kc_1(A)=1$: $\ch_2(A) \leq \tfrac{1}{32}$ so $\ch_2(A) \leq 0$.  Thus $\Omega_X$ is $\sigma_{tK}$-semistable for $t^2>0$.
    \item $Kc_1(A)=2$: $\ch_2(A) \leq \tfrac{1}{8}$ so $\ch_2(A) \leq 0$.  Thus $\Omega_X$ is $\sigma_{tK}$-semistable for $t^2>0$.
    \item $Kc_1(A)=3$: $\ch_2(A) \leq \tfrac{9}{32}$ so $\ch_2(A) \leq 0$.  Thus $\Omega_X$ is $\sigma_{tK}$-semistable for $t^2>0$.
    \item $Kc_1(A)=4$: $\ch_2(A) \leq \tfrac{1}{2}$.  Thus $\Omega_X$ is $\sigma_{tK}$-semistable for $t^2>\tfrac18$.
    \item $Kc_1(A)=5$: $\ch_2(A) \leq \tfrac{25}{32}$ so $\ch_2(A) \leq \tfrac12$.  Thus $\Omega_X$ is $\sigma_{tK}$-semistable for $t^2>\tfrac16$.
    \item $Kc_1(A)=6$: $\ch_2(A) \leq \tfrac{9}{8}$ so $\ch_2(A) \leq 1$.  Thus $\Omega_X$ is $\sigma_{tK}$-semistable for $t^2>\tfrac12$.
\end{itemize}
Therefore, for $1 \leq Kc_1(A) \leq 6$, we can conclude that $\Omega_X$ is always $\sigma_{tK}$-semistable as long as $t \geq \tfrac12$.

\noindent
\textbf{If $Kc_1(A) =7$:} We divide further into two cases according to whether $A$ is $\mu_K$-semistable.  
\begin{itemize}
    \item If $A$ is $\mu_K$-semistable, then the Hodge index theorem implies that $c_1(A)^2\leq \frac{49}{8}$.  On the other hand, by Remark \ref{rem:congmod2}, $c_1(A)^2\equiv Kc_1(A)$ mod $2$ and so $c_1(A)^2$ is odd, i.e.\ $c_1(A)^2 \leq 5$.  The BG inequality now gives $\ch_2(A)\leq \frac{c_1(A)^2}{4}\leq \frac{5}{4}$.  Since $\ch_2(A)$ is a half integer, we must have $\ch_2(A) \leq \tfrac12$.
    \item If $A$ is not $\mu_K$-semistable, then $A$ has exactly two $\mu_K$-HN factors, say $A_1, A_2$ where $\mu_K(A_1) > \mu_K (A_2)$.  Since we must have $Kc_1(A_1)>0$, by Lemma \ref{lem:AG53-52-1} we  have $\ch_2(A_1) \leq 0$. Also, since $A_1$ is a subsheaf of $A$ which in turn is a subsheaf of $\Omega_X$, the $\mu_K$-semistability of $\Omega_X$ implies $\mu_K(A_1)=Kc_1(A_1) \leq \mu_K(\Omega_K)=4$, and so $Kc_1(A_2)=\mu_K(A_2) \leq 3$.  By the Hodge index theorem and the BG inequality, we have $\ch_2(A_2) \leq \frac{ (Kc_1(A_2))^2 }{2K^2} \leq \tfrac{9}{16}$ and so $\ch_2(A_2) \leq \tfrac12$.  Thus $\ch_2(A) \leq 0$.
\end{itemize}
Overall, we have $\ch_2(A) \leq \tfrac12$ when $Kc_1(A)=7$.  From the wall equation we now have
\[
t^2 \leq \frac{1/2}{8-7}=\frac12.
\]
Hence when $Kc_1(A)=7$, the cotangent bundle $\Omega_X$ is $\sigma_{tK}$-semistable for all $t^2\geq \frac{1}{2}$.

Note that since $\ch_0(Q)=0$, we must have $Kc_1(Q)>0$ or else $\rho_\omega (Q)=\infty$, contradicting our assumption that $Q$ is a quotient of $\Omega_X$ in $\mathcal{B}_K$ at a wall.  Hence we are done.
\end{proof}

\begin{prop}
\label{prop:destsubsheafeven}
Suppose $X$ is a fake quadric of even type. Then there is no destabilizing sequence
\[
0\to A \to \Omega_X\to Q\to 0
\]
in $\B_{K}$ where $A$ is a rank $2$ subsheaf of $\Omega_X$.
\end{prop}


\begin{proof}
Assuming that there is a destabilizing short exact sequence. 
\[
0\to A\to \Omega_X\to Q\to 0
\]
in $\B_{K}$ for some $t>0$. We may assume $Q$ is pure, otherwise we can take the pure quotient and consider a wall with larger $t$ value. For $A\in \B_{K}$, one must have $K\ch_1(A)>0$. Denoting $\ch_1(Q)\sim xQ_1+yQ_2$, this implies that $8-(2x+2y)>0$. Combining the discription of $\text{Eff}(X)$ in Proposition \ref{prop:ANNeven}, we have the following possibilities for $c_1(Q)$:
\begin{equation}
\label{eq:effcurve}
\begin{split}
&x=1, y=0; x=0, y=1;\\
&x=2, y=0; x=0, y=2; x=1, y=1;\\
&x=2, y=1; x=1, y=2; x=3, y=0; x=0, y=3.
\end{split}
\end{equation}
Since $X$ has no negative curves, Lemma \ref{lem:Miyrmkcoro} implies that $\ch_2(A)\leq \frac{c_1(A)^2}{6}$, provided that $\ch_1(A)$ is nef.
Since $\rho_{tK_X, 0}(\Omega_X)=\rho_{tK_X, 0}(A)$, we have 
\[
\begin{split}
&\frac{\ch_2(\Omega_X)-t^2K_X^2}{K_X^2}=\frac{\ch_2(A)-t^2K_X^2}{K_X\cdot\ch_1(A)}\leq \frac{\frac{c_1(A)^2}{6}-t^2K_X^2}{K_X\cdot\ch_1(A)}\\
&\Leftrightarrow -\frac{8t^2}{8}\leq \frac{\frac{1}{3}(2-x)(2-y)-8t^2}{2(2-x)+2(2-y)}\\
&\Leftrightarrow \frac{1}{3}(2-x)(2-y)\geq \frac{8t^2}{4}(x+y).
\end{split}
\]
Among the cases for which $c_1(A)$ is nef, the only possible choices for $(x, y)$ are $(1, 0)$, $(0, 1)$, $(1, 1)$, the two nonnef cases are $(3, 0)$, $(0, 3)$. 

First we consider the case $(1, 1)$. Then Lemma \ref{lem:Miyrmkcoro} implies that $\ch_2(A)\leq \frac{1}{3}$. Since $\ch_2(Q)<0$, we have $\ch_2(A)>0$, this contradicts that $\ch_2(A)\in \frac{1}{2}\mathbb{Z}$.

Next we consider the cases $(3, 0)$, $(0, 3)$, these are the cases Lemma \ref{lem:Miyrmkcoro} does not apply. Assuming that $A$ is $\mu$-semistable. Then Bogomolov-Gieseker inequality implies that $\frac{1}{2}\leq \ch_2(A)\leq \frac{1}{4}\ch_1(A)^2=\frac{1}{2}(2-x)(2-y)$. Hence these cases cannot happen. Assume that $A$ is not $\mu$-semistable. Let 
\[
0\to M\to A\to N\to 0
\]
 be the short exact sequence. such that $M$ and $N$ being the HN-factors of $A$. 
 Then 
 \[
 \begin{split}
 \ch_2(A)&=\ch_2(M)+\ch_2(N)\leq \frac{1}{2}(\ch_1(M)^2+\ch_1(N)^2)\leq \frac{1}{2K^2}((K\cdot\ch_1(M))^2+(K\cdot\ch_1(N))^2)\\
 &\leq \frac{1}{2K^2}(K\cdot\ch_1(M)+K\cdot\ch_1(N))^2=\frac{1}{2K^2}(K\cdot\ch_1(A))^2=\frac{1}{2K^2}(8-(2x+2y))^2.
 \end{split}
 \]
 Hence these cases cannot happen either. We only need to consider $(1, 0)$ and $(0, 1)$.
 
Assume that $Q_1$ is effective and $c_1(Q)=Q_1$. The case of $c_1(Q)=Q_2$ is analogous. 
Let $C=\text{supp}(Q)$, then $Q_1$ being a primitive generator implies that $C$ is irreducible and reduced. Furthermore, the map $\Omega_X\to Q$ factors through $\Omega_X|_C\to Q$. Note that the conormal sheaf $I_C/I_C^2\simeq \cO_X(-C)\otimes \cO_C= \cO_C(-C)$, which is locally free of rank $1$ on $C$. Since $C$ is integral, the conormal exact sequence is also left exact. Denote the image of $\cO_C(-C)\to \Omega_X|_C\to Q$ by $N_1$, and consider the following diagram:
\[
\begin{tikzcd}
   &  0\ar{d}  &  0\ar{d}  &  0\ar{d}  &\\
 0\ar{r} & K_1\ar{d}\ar{r} & K\ar{d}\ar{r}{e} & K_2\ar{d}{h}\ar{r} & 0 \\
0\ar{r}&  \cO_C(-C)\ar{r}{c}\ar{d}{s} & \Omega_X|_C \ar{r}{a}\ar{d}{p}& \Omega_C\ar{r}\ar{d}{q} & 0 \\
0\ar{r}&  N_1\ar{r}{i}\ar{d} & Q\ar{r}{g}\ar{d}& N_2\ar{d}\ar{r} & 0\\
& 0 & 0 & 0 & 
\end{tikzcd}
\]
We have 
$\ch_2(\cO_C(-C))=0$ and $\ch_2(\Omega_C)=2$. Since $Q$ is pure, if $N_1\neq 0$ then $N_1$ is pure and $\cO_C(-C)\simeq N_1$. This implies $\ch_2(N_1)=0$ and hence $\ch_2(Q)\geq 0$, which contradicts $\ch_2(Q)< 0$. Hence $N_1=0$ and $Q\simeq N_2$. 
Denoting the kernel of $\Omega_X\to \Omega_C$ by $B$, 
then $A$ fits into a short exact sequence. 
\begin{equation}
\label{eq:7.6_1}
0\to B\to A\to K_2\to 0, 
\end{equation}
By Lemma \ref{lem:Miyrmkcoro}, we have $\ch_2(A)\leq \frac{c_1(A)^2}{6}=\frac{2}{3}$. On the other hand $\ch_2(B)=-2$. Hence $\text{length}(K_2)\leq\frac{8}{3}$, which implies that $\text{length}(K_2)\leq 2$. This forces $\ch_2(N_2)\geq 0$, hence this case cannot happen either.
\end{proof}

We now consider the family of odd fake quadrics constructed in \cite{cat25}. 
We briefly review the construction. Consider a family of surfaces $S$ which is isogenous to a higher product of unmixed type, i.e.
\[
S=C_1\times C_2/G,
\]
where $C_i$'s are smooth curve with $g_i=g(C_i)\geq 2$, and $G$ is a finite group acting
faithfully on $C_1$ and on $C_2$, such that the diagonal action $g(x, y):=(gx, gy)$
on $C_1\times C_2$ is free. In particular, $S$ has no negative curve. Take the group $G\simeq A_5$, 
and $C_1$ is the $G$-covering of $\mathbb{P}^1$
branched on three points with Hurwitz triple
$(2, 4)(3, 5)$, $(2, 1, 3, 4, 5)$, $(1, 2, 3, 4, 5)$, $C_2$ is the $G$-covering of $\mathbb{P}^1$
branched on four points with Hurwitz quadruple
 $(1, 2, 3)$, $(3, 4, 5)$, $(4, 3, 2)$, $(2, 1, 5)$. 
We apply Riemann-Hurwitz formula
\[
2g(C_i)-2=\deg(f)(2g(\mathbb{P}^1)-2)+\sum_{p\in C_i}(e_p-1).
\]
Since $G$ acts faithfully on $C_i$, $\deg(f)=|G|=60$. The ramification type implies that $g(C_1)=4$ and $g(C_2)=21$. Denoting the general fiber of the  projection $X\to C_i/G$ by $F_i$. Using the ramification formular and the notation of \cite{cat25}, $F_1=10\Phi_1$, $F_2=3\Phi_2$ and $K_S=\Phi_1+2\Phi_2$. Since $F_1\cdot F_2=60$, we have $\Phi_1\cdot\Phi_2=2$ and $\Phi_i^2=0$. Hence $\Phi_1=Q_1+Q_2$ and $\Phi_2=Q_1-Q_2$. Denote $Y:=C_1\times C_2$, and consider the diagram: 
\[
\begin{tikzcd}
Y\ar{r}{\pi}\ar{d}{p_i}& X\ar{d}\\
C_i\ar{r}& C_i/G
\end{tikzcd}
\]
Then $\pi^*F_1=60C_2$ and $\pi^*F_2=60C_1$.

\begin{prop}
\label{prop:destabsubsheafodd}
Let $X$ be the fake quadric of odd type as constructed in \cite{cat25}. Then there is no destabilizing sequence
\begin{equation}
\label{eq:dessesodd}
0\to A\to \Omega_X\to Q\to 0
\end{equation}
in $\B_{K}$ where $A$ is a rank 2 subsheaf of $\Omega_X$.
\end{prop}
\begin{proof}
Assuming that there is a destabilizing short exact sequence. 
\[
0\to A\to \Omega_X\to Q\to 0
\]
in $\B_{K}$. For the same reason as in Proposition \ref{prop:destsubsheafeven}, we may assume $Q$ is pure. 
Denoting $\ch_1(Q)\sim xQ_1+yQ_2$. Since there is no negative curve on $X$, we have $x^2-y^2\geq 0$. Also, $A\in \B_{K}$ implies that $0<K_X\cdot \ch_1(Q)<8$, which is equivalent to $0<3x+y<8$. Hence, the possible values for $x$ and $y$ are: 
\[
\begin{split}
&x=1, y=-1, 0, 1,\\
&x=2, y=-2, -1, 0, 1,\\
&x=3, y=-3, -2.
\end{split}
\]
Since $\rho_{tK}(\Omega_X)=\rho_{tK}(Q)$, we have 
$\ch_2(Q)<0$, hence $\ch_2(Q)\leq -\frac{1}{2}$.
By Lemma \ref{lem:Miyrmkcoro}, we have 
\[
\frac{c_1(A)^2}{6}\geq \ch_2(A)\geq 1/2,
\]
provided that $\ch_1(A)$ is nef. This implies that the only possible cases are $(1, 0)$, $(1, -1)$, $(2, 1)$, $(3, -2)$ and $(3, -3)$.

First we consider the cases $(2, 1)$, $(3, -2)$ and $(3, -3)$. Similar argument as in Proposition \ref{prop:destsubsheafeven} implies that these cases cannot happen. Indeed, assuming that $A$ is $\mu$-semistable. Then Bogomolov-Gieseker inequality implies that $\frac{1}{2}\leq \ch_2(A)\leq \frac{1}{4}\ch_1(A)^2$. Hence these cases cannot happen. Assume that $A$ is not $\mu$-semistable. Breaking $A$ into the HN-factors $M$ and $N$, then we have 
\[
\ch_2(A)\leq \frac{1}{2K^2}(3(3-x)-(y+1))^2.
\]
 Hence these three cases cannot happen either, and we are left with two cases: (i) $(x, y)=(1, -1)$, (ii) $(x, y)=(1, 0)$.

Case (i) Assume that $(1, -1)$ is effective, any curve in the linear system has self-intersection $0$. Then the argument in Proposition \ref{prop:destsubsheafeven} implies that \ref{eq:dessesodd} cannot be a destabilizing short exact sequence. for $\Omega_X$. Indeed, assume that $C=\text{supp}(Q)\in |Q_1-Q_2|$, then $C$ is irreducible and reduced. In this case, we have the same diagram as in Proposition \ref{prop:destsubsheafeven}. If $N_1\neq 0$, $\ch_2(\cO_C(-C))=0$, hence $\ch_2(Q)\geq 0$ which contradict $\ch_2(Q)\leq -\frac{1}{2}$. This forces $N_1=0$, but in this case $Q\simeq N_2$. Since $\ch_2(\Omega_C)=2$ and $\text{length}(K_2)\leq 2$, we again obtain a contradiction. 


Lastly we consider case (ii). From now on we assume $X$ is the fake quadric of odd type constructed in \cite{cat25}. Assuming that $Q_1$ is effective. 
 The destabilizing short exact sequence \ref{eq:dessesodd} induces a short exact sequence where the phases are preserved:
\begin{equation}
\label{eq:destaboddY}
0\to \pi^*A\to \Omega_Y\to \pi^*Q\to 0
\end{equation}
for $\Omega_Y$. We denote the image of the composition $p^*\Omega_{C_1}\to\Omega_Y\to\pi^*Q$ by $N_1$, and form the following diagram:
\[
\begin{tikzcd}
   &  0\ar{d}  &  0\ar{d}  &  0\ar{d}  &\\
 0\ar{r} & K_1\ar{d}\ar{r} & \pi^*A\ar{d}\ar{r} & K_2\ar{d}\ar{r} & 0 \\
0\ar{r}&  p_1^*\Omega_{C_1}\ar{r}\ar{d} & \Omega_Y \ar{r}\ar{d}& p_2^*\Omega_{C_2}\ar{r}\ar{d} & 0 \\
0\ar{r}&  N_1\ar{r}\ar{d} & \pi^*Q\ar{r}\ar{d}& N_2\ar{d}\ar{r} & 0\\
& 0 & 0 & 0 & 
\end{tikzcd}
\]
Denoting the two fiber classes by $H_1:=C_1\times \{p\}$, $H_2:=\{q\}\times C_2$. Since $c_1(Q)=Q_1=\frac{1}{2}(\Phi_1+\Phi_2)$, one has \[
c_1(\pi^*Q)=\frac{1}{2}(\pi^*\Phi_1+\pi^*\Phi_2)=\frac{1}{2}(\frac{1}{10}\pi^*F_1+\frac{1}{3}\pi^*F_2)=3H_2+10H_1.
\]
Let $C=\supp(\pi^*Q)$, then $C$ is a pure reduced one dimensional subscheme. If $N_1\neq 0$, since it is $G$-equivariant, we can denote its support by $C'=aH_1+bH_2$, which is again a pure curve. Then 
$N_1\simeq p_1^*\Omega_{C_1}|_{C'}$. Hence 
\[
\begin{split}
\ch_2(N_1)=\ch_2(p_1^*\Omega_{C_1}|_{C'})&=(2g(C_1)-2)C_2\cdot C'-\frac{1}{2}C'^2.\\
&=(2g(C_1)-2)a-ab.
\end{split}
\]
Since $g(C_1)=4$, $a\geq 0$, $b\leq 3$, we have $\ch_2(N_1)\geq 0$. 

Applying similar argument to the pure quotient of $N_2$, denoting its support by $C''=cH_1+dH_2$. We find 
\[
\ch_2(N_2)\geq \ch_2(p_2^*\Omega_{C_2}|_{C''})=(2g(C_2)-2)d-cd.
\]
Since $g(C_2)=21$, $d\geq 0$, $c\leq 10$, we have $\ch_2(N_2)\geq 0$. Combine the results, we have $\ch_2(\pi^*Q)\geq 0$, hence \ref{eq:destaboddY} cannot be a destabilizing short exact sequence. for $\Omega_Y$. 
\end{proof}

\begin{cor}
\label{cor:fakequadricmain}Let $X$ be a fake quadric of even type, or a fake quadric of odd type constructed in \cite{cat25}.  Then $\Omega_X$ is $\sigma_{tK}$-semistable for all $t^2\geq \frac{1}{3}$.
\end{cor}
\begin{proof}
By the proof of Theorem \ref{thm:wallfakequadric}, we have for $t^2\geq \frac{1}{3}$, the destabilizing object $A$ has rank equals to $2$. Then Proposition \ref{prop:destsubsheafeven} and Proposition \ref{prop:destabsubsheafodd} shows that the rank $2$ destabilizing subobject does not exist and the conclusion follows.   
\end{proof}

\paragraph If $C_1, C_2$ are smooth projective curves of genera $g_1, g_2 \geq 2$, respectively, then a suitable finite abelian group $G$  acting freely on  $C_1 \times C_2$ by the product action can produce  minimal surfaces  of general type with $p_g=0, K^2_X=8$.  Such a construction is due to Beauville, and minimal surfaces of general type arising this way were classified by Bauer and Catanese in \cite{bauer2003some}. This will be a direction for future work.



\bibliography{refs}{}
\bibliographystyle{plain}

\end{document}